\documentclass[12pt]{amsart}
\usepackage{amsfonts, amssymb, amsmath, amsthm, eucal, latexsym, array,amsbsy,latexsym}
\usepackage{fullpage}
\usepackage{verbatim}

\usepackage{color}
\definecolor{darkblue}{rgb}{0,0,.5}
\definecolor{darkgreen}{rgb}{.2,0.5,.2}
\usepackage[colorlinks=true,linkcolor=darkblue,urlcolor=violet,citecolor=magenta]{hyperref}

\numberwithin{equation}{section}

\input{cyracc.def}

\newtheorem{thm}{Theorem}[section]
\newtheorem{conj}[thm]{Conjecture}
\newtheorem{lm}[thm]{Lemma}
\newtheorem{cl}[thm]{Corollary}
\newtheorem{prop}[thm]{Proposition}

\theoremstyle{remark}
\newtheorem{ex}[thm]{Example}
\newtheorem{rmk}[thm]{Remark}

\theoremstyle{definition}
\newtheorem{df}[thm]{Definition}

\newcommand{\gt}{\mathfrak}

\newcommand{\Hom}{\mathrm{Hom}}

\newcommand{\ad}{\mathrm{ad}}

\newcommand {\ess}{{\mathcal Es}}
\newcommand {\bv}{{\boldsymbol{v}}}

\newcommand {\cS}{{\mathcal S}}

\newcommand {\mK}{{\mathbb C}}

\renewcommand{\le}{\leqslant}
\renewcommand{\ge}{\geqslant}

\newfam\eusfam%
\font\euszw=eusm10 scaled 1200%
\font\eusac=eusm7 scaled 1200%
\font\eusacc=eusm7 scaled 1000%
\textfont\eusfam=\euszw\scriptfont\eusfam=\eusac%
\scriptscriptfont\eusfam=\eusacc%
\newcommand{\non}{\nonumber}

\newcommand{\ot}{\otimes}
\newcommand{\la}{\lambda}

\newcommand{\La}{\Lambda}
\newcommand{\De}{\Delta}
\newcommand{\al}{\alpha}
\newcommand{\be}{\beta}
\newcommand{\ga}{\gamma}

\newcommand{\si}{\sigma}

\newcommand{\de}{\delta}
\newcommand{\ze}{\zeta}

\newcommand{\ve}{\varepsilon}
\newcommand{\ts}{\,}

\newcommand{\tss}{\hspace{1pt}}

\newcommand{\U}{{\mathcal U}}

\newcommand{\CC}{\mathbb{C}\tss}

\newcommand{\ZZ}{\mathbb{Z}\tss}

\newcommand{\Y}{ {\rm Y}}

\newcommand{\Zr}{{\rm Z}}
\newcommand{\gl}{\mathfrak{gl}}

\newcommand{\oa}{\mathfrak{o}}
\newcommand{\spa}{\mathfrak{sp}}

\newcommand{\ggot}{\mathfrak{g}}
\newcommand{\rgot}{\mathfrak{r}}
\newcommand{\h}{\mathfrak h}
\newcommand{\n}{\mathfrak n}

\newcommand{\sgn}{ {\rm sgn}\ts}

\newcommand{\fand}{\quad\text{and}\quad}
\newcommand{\Fand}{\qquad\text{and}\qquad}

\newcommand{\bth}{\begin{thm}}
\renewcommand{\eth}{\end{thm}}
\newcommand{\bpr}{\begin{prop}}
\newcommand{\epr}{\end{prop}}
\newcommand{\ble}{\begin{lm}}
\newcommand{\ele}{\end{lm}}
\newcommand{\bco}{\begin{cl}}
\newcommand{\eco}{\end{cl}}
\newcommand{\bex}{\begin{ex}}
\newcommand{\eex}{\end{ex}}
\newcommand{\bre}{\begin{rmk}}
\newcommand{\ere}{\end{rmk}}
\newcommand{\bcj}{\begin{conj}}
\newcommand{\ecj}{\end{conj}}

\newtheorem*{mthma}{Theorem A}
\newtheorem*{mthmb}{Theorem B}

\newcommand{\bal}{\begin{aligned}}
\newcommand{\eal}{\end{aligned}}
\newcommand{\beq}{\begin{equation}}
\newcommand{\eeq}{\end{equation}}
\newcommand{\ben}{\begin{equation*}}
\newcommand{\een}{\end{equation*}}

\newcommand{\bpf}{\begin{proof}}
\newcommand{\epf}{\end{proof}}

\def\beql#1{\begin{equation}\label{#1}}

\begin{document}
\hfill {\scriptsize  November 3, 2019}
\vskip1ex

\title{Monomial bases and branching rules}
\author[A.\,Molev]{Alexander Molev}
\address[A.\,Molev]
{School of Mathematics and Statistics,
University of Sydney,
NSW 2006, Australia}
\email{alexander.molev@sydney.edu.au}
\author[O.\,Yakimova]{Oksana Yakimova}
\address[O.\,Yakimova]{Universit\"at zu K\"oln,
Mathematisches Institut, Weyertal 86-90, 50931 K\"oln, Deutschland}
\email{oksana.yakimova@uni-jena.de}
\curraddr{Institut f\"ur Mathematik, Friedrich-Schiller-Universit\"at Jena, Jena, 07737, Deutschland}
\thanks{The second author is 
funded by the Deutsche Forschungsgemeinschaft (DFG, German Research Foundation) --- project  numbers 330450448, 404144169.}
\keywords{}
\begin{abstract}
Following a question of Vinberg, a general method to construct monomial bases
in finite-dimensional irreducible representations of a reductive Lie algebra $\ggot$
was developed in a series of papers by Feigin, Fourier, and Littelmann.
Relying on this method,
we construct monomial bases of multiplicity spaces
associated with the restriction of the representation to a reductive subalgebra $\ggot_0\subset\ggot$.
As an application, we produce  new monomial bases for representations of the 
symplectic Lie algebra associated with a natural chain of subalgebras.
One of our bases
is related via  a triangular transition matrix to
a suitably modified version of the basis constructed earlier by the first author.
In type {\sf A}, our approach shows that 
the Gelfand--Tsetlin basis and the canonical basis of Lusztig have a
common PBW-parameterisation. This implies that the transition matrix between them is triangular.
We show also that  a  similar relationship holds for the Gelfand--Tsetlin and the Littelmann bases in
type {\sf A}.
\end{abstract}
\maketitle

\section*{Introduction}

A general method to construct monomial bases
in finite-dimensional irreducible representations of a reductive Lie algebra $\ggot$
has been developed in a series of papers by
E.~Feigin, G.~Fourier,  and P.~Littelmann~\cite{ffl:pbwa,ffl:pbwc,ffl:fm} following a question and initial
examples of E.~Vinberg. In accordance with this method, one chooses a triangular decomposition
$\gt g=\gt n^-\oplus\gt h\oplus\gt n^+$ and a basis $\{f_1,\ldots,f_N\}$
of the nilpotent Lie algebra $\gt n^-$ consisting
of root vectors. Let $V(\la)$ be a finite-dimensional irreducible $\ggot$-module
and let $v_{\lambda}\in V(\la)$ be a highest weight vector.
By introducing special orderings on monomials in the basis elements $f_i$
it is possible to specify conditions on the powers $\al_i$ so that the vectors
\ben
f_1^{\al_1}\dots f_N^{\al_N}\tss v_{\lambda}
\een
form a basis of $V(\la)$. Such conditions are given in an explicit form
for types {\sf A} and {\sf C} in \cite{ffl:pbwa} and \cite{ffl:pbwc}, respectively.
A unified approach is presented in \cite{ffl:fm}.

One of the features of the initial solutions \cite{ffl:pbwa, ffl:pbwc}  is that
a {\it homogeneous} order on the monomials was used, which  means that the degrees are compared first.
In such a setup, the sequence  of factors is not significant.
By now there is a  tremendous development in the area,
with both geometric and combinatorial applications,  and numerous variations have been studied,
see e.g.  \cite{SP, xfp} and references therein. Of particular interesest and importance are connections
with the Littelmann bases \cite{l:cc} and with the PBW-type versions of the canonical basis
of Lusztig \cite{L1,L2}, see
\cite[Sect.~11\&12]{xfp}.

Our goal in this paper is to 
adjust the FFLV method to construct bases
of the multiplicity spaces associated with the restriction of $V(\la)$
to a reductive subalgebra $\ggot_0$.
Given a finite-dimensional irreducible $\ggot_0$-module $V'(\mu)$,
the corresponding
{\em multiplicity space} is defined by
\ben
U(\la,\mu)=\Hom_{\ggot_0}\big(V'(\mu),V(\la)\big).
\een
Note that $U(\la,\mu)$ is isomorphic to the subspace $V(\la)^+_{\mu}$
of $\ggot_0$-highest weight vectors in $V(\la)$ of weight $\mu$ and we have
a vector space decomposition
\beql{dectp}
V(\la)\cong \underset{\mu}{\oplus}\ts\ts V(\la)^+_{\mu}\otimes V'(\mu).
\eeq
Hence, if some bases of the spaces $V(\la)^+_{\mu}$ and $V'(\mu)$ are produced, then
the decomposition \eqref{dectp} yields the
natural tensor product basis of $V(\la)$.
The celebrated Gelfand--Tsetlin bases \cite{gt:fdu, gt:fdo}
for representations of the general linear
and orthogonal Lie algebras are obtained by
iterating this procedure and applying it
to the subalgebras of the chains
\ben
\gl_1\subset \gl_2\subset\dots\subset \gl_n\Fand
\oa_2\subset \oa_3\subset\dots\subset \oa_N.
\een

The multiplicity spaces $V(\la)^+_{\mu}$ corresponding to the pairs
of orthogonal and symplectic Lie algebras $\oa_{N-2}\subset \oa_{N}$
and $\spa_{2n-2}\subset \spa_{2n}$ turned out to carry representations
of certain quantum algebras originally
introduced by Olshanski~\cite{o:ty} and which are known as {\em twisted Yangians}.
The Yangian representation theory together with the theory of {\em Mickelsson algebras}
developed in the work by Zhelobenko~\cite{z:gz, z:ep, z:it} have lead to a construction
of bases of the Gelfand--Tsetlin type for representations
of the orthogonal and symplectic Lie algebras; see review paper~\cite{m:gtb}
and book \cite[Ch.~9]{m:yc} for a detailed exposition of these results, as well as a discussion
of various approaches to constructions of Gelfand--Tsetlin-type bases
in the literature.

The Zhelobenko theory allows one to describe the multiplicity spaces $V(\la)^+_{\mu}$
corresponding to the pair $\ggot_0\subset\ggot$ as linear spans of
{\em lowering operators} obtained via the action of the {\em extremal projector} $p$
associated with the Lie algebra $\ggot_0$.
Our main general result provides precise choices
of those operators to form a basis of $V(\la)^+_{\mu}$. These choices are made in the spirit
of the FFLV method and rely on some special monomial order.
In more detail, we will assume that
$\ggot_0\subset\ggot$ is a reductive subalgebra normalised
by $\gt h$.
Then $\gt g_0$ inherits the triangular decomposition
$\gt g_0=\gt n_0^+\oplus\gt h_0\oplus\gt n_0^-$ with
$\gt n_0^{\pm}=\gt n^{\pm}\cap\gt g_0$ and $\gt h_0=\gt h\cap\gt g_0$.
Let $\gt n^-=\gt n_0^-\oplus \gt r$
be an $\gt h$-stable vector space decomposition. We describe
a family of admissible monomials $m\in \U(\rgot)$ 
such that
the elements $p\tss m\tss v_{\lambda}$ form a basis of the multiplicity space $V(\la)^+_{\mu}$.

In order  to obtain a basis for $V(\la)$ inductively, it suffices to produce first a basis
for a quotient of $V(\lambda)$ that is isomorphic  to $U(\lambda,\mu)$ in some natural way.
This idea is used  in  \cite{l:cc} and \cite{Gor}.  In the latter, an answer to Vinberg's question
for the orthogonal Lie algebra is given.
We formalise the method that can be called the ``FFLV-branching'' in Section~\ref{fflv-br}  and as an application produce  a new answer to Vinberg's question in type {\sf C} in Section~\ref{br}.

Recall that finite-dimensional irreducible representations of $\gl_n$
are parameterised by their highest weights $\la=(\la_1,\dots,\la_n)$ which
are $n$-tuples of complex numbers satisfying the conditions
$\la_i-\la_{i+1}\in\ZZ_+$ for all $i=1,\dots,n-1$.
Later on we will need to work simultaneously with irreducible highest weight
representations of $\spa_{2n}$ and $\gl_n$. To avoid a confusion we
will denote such representations of $\gl_n$ by $L(\lambda)$ while keeping the notation
$V(\lambda)$ in the context of general complex reductive Lie algebras
and in the particular case of the symplectic Lie algebras.
Thus $L(\lambda)$
is generated by a nonzero vector $v_{\lambda}$
such that
\begin{alignat}{2}
\label{hwglnrep}
E_{ij}\ts v_{\lambda}&=0 \qquad &&\text{for} \quad
1\leqslant i<j\leqslant n, \qquad \text{and}\\
E_{ii}\ts v_{\lambda}&=\la_i\ts v_{\lambda} \qquad &&\text{for}
\quad 1\leqslant i\leqslant n,
\non
\end{alignat}
where the $E_{ij}$ denote the standard basis elements of $\gl_n$.
A {\em Gelfand--Tsetlin pattern}
$\Lambda$ associated with
$\lambda$ is an array of row vectors
\begin{align}
&\qquad\lambda^{}_{n\tss 1}\qquad\lambda^{}_{n\tss 2}
\qquad\qquad\cdots\qquad\qquad\lambda^{}_{n\tss n}\non\\
&\qquad\qquad\lambda^{}_{n-1\ts 1}\qquad\ \ \cdots\ \
\ \ \qquad\lambda^{}_{n-1\ts n-1}\non\\
&\quad\qquad\qquad\cdots\qquad\cdots\qquad\cdots\non\\
&\quad\qquad\qquad\qquad\lambda^{}_{2\tss 1}\qquad\lambda^{}_{2\tss 2}\non\\
&\quad\qquad\qquad\qquad\qquad\lambda^{}_{1\tss 1}  \non
\end{align}
where the upper row coincides with $\lambda$ and
the following conditions hold
\beql{aconl}
\lambda^{}_{k\tss i}-\lambda^{}_{k-1\ts i}\in\ZZ_+,\qquad
\lambda^{}_{k-1\ts i}-\lambda^{}_{k\ts i+1}\in\ZZ_+,\qquad
i=1,\dots,k-1
\eeq
for each $k=2,\dots,n$.

Let $\{\xi_\Lambda\}$ be the Gelfand--Tsetlin basis for $L(\lambda)$, see
Section~\ref{sec:rgt} for its detailed description.

\begin{mthma} {\em
Let $\{\pi_\Lambda\}$ be the set of vectors
\ben
\pi_{\La}=E_{2\tss 1}^{\la^{}_{2\tss 1}-\la^{}_{1\tss 1}}
E_{3\tss 1}^{\la^{}_{3\tss 1}-\la^{}_{2\tss 1}}  E_{3\tss 2}^{\la^{}_{3\tss 2}-\la^{}_{2\tss 2}}
\dots E_{n\tss 1}^{\la^{}_{n\tss 1}-\la^{}_{n-1\ts 1}}
\dots E_{n\ts n-1}^{\la^{}_{n\tss n-1}-\la^{}_{n-1\tss n-1}}\tss v_{\lambda}\, ,
\een
where $\Lambda$ runs over  all Gelfand--Tsetlin patterns associated with
$\lambda$. Then $\{\pi_\Lambda\}$  is a PBW-para\-me\-te\-risation of  $\{\xi_\Lambda\}$, i.e.,
there is an order
$\succ$ on $\{\Lambda\}$ such that
$$
\xi_{\Lambda} = \sum_{\La'\succcurlyeq \La} d_{\La,\La'}\tss \pi^{}_{\La'}
$$
with $d_{\Lambda,\Lambda'}\in\mK$ and $d_{\Lambda,\Lambda}\ne 0$.
In particular,  $\{\pi_\Lambda\}$ is a basis for $L(\la)$.}
\end{mthma}

Theorem A
is proven in Section~\ref{sec:rgt} by essentially repeating the argument used by  Zhelobenko in \cite[Theorem~7]{z:gz} and \cite[Lemma~2]{z:ep}. We also indicate briefly how it follows from the FFLV-theory.

The basis described in Theorem~A is not new.
With a slightly different, but combinatorially equivalent, description it appeared in \cite{Markus-dis}
as a PBW-parameterisation of the canonical basis of Lusztig \cite{L1,L2}.  Therefore the theorem provides
a link between the Gelfand--Tsetlin and the canonical bases, see 
Corollary~\ref{cor:tramatwo}.
The same basis is described  in \cite[Theorem~2.6]{mah}.

We will regard the symplectic Lie algebra $\spa_{2n}$ as a subalgebra
of $\gl_{2n}$ and we will number
the rows and columns of $2n\times 2n$ matrices with
the indices $-n,\dots,-1,1,\dots,n$. Accordingly,
the zero value will be omitted in the summation or product formulas.
The Lie algebra
$\spa_{2n}$ is spanned by the elements $F_{ij}$ with
$-n\leqslant i,j\leqslant n$, defined by
\beql{fijroot}
F_{ij}=E_{ij}-\sgn i\ts\ts \sgn j\ts E_{-j,-i}.
\eeq
For any $n$-tuple of nonpositive integers
$\la=(\la_1,\dots,\la_n)$ satisfying the conditions
\ben
\la_1\geqslant\la_2\geqslant\dots\geqslant\la_n
\een
the finite-dimensional irreducible representation $V(\la)$ of
the Lie algebra $\spa_{2n}$ with the highest weight $\la$ is generated by a nonzero vector $v_{\lambda}$
such that
\begin{alignat}{2}
\label{hwrepbcd}
F_{ij}\ts v_{\lambda}&=0 \qquad &&\text{for} \quad
-n\leqslant i<j\leqslant n, \qquad \text{and}\\
F_{ii}\ts v_{\lambda}&=\la_i\ts v_{\lambda}
\qquad &&\text{for} \quad 1\leqslant i\leqslant n.
\non
\end{alignat}

Define a {\em type  {\sf C}
pattern} $\La$
associated with
$\la$ as an array of the form
\begin{align}
\quad&\qquad\la^{}_{n\tss 1}\qquad\la^{}_{n\tss 2}\qquad\la^{}_{n\tss 3}
\qquad\quad\cdots\quad\qquad\la^{}_{n\tss n}\non\\
&\la'_{n\tss 1}\qquad\  \la'_{n\tss 2}\qquad \la'_{n\tss 3}
\qquad\quad\cdots\quad\qquad\la'_{n\tss n}\non\\
\ &\qquad\la^{}_{n-1\ts 1}\quad\la^{}_{n-1\ts 2}\qquad\cdots
\qquad\la^{}_{n-1\ts n-1}\non\\
&\la'_{n-1\ts 1}\ \quad\la'_{n-1\ts 2}
\qquad\cdots\qquad\la'_{n-1\ts n-1}\non\\
&\qquad\cdots\qquad\cdots\qquad\cdots\non\\
\quad&\qquad\la^{}_{1\tss 1}\non\\
&\la'_{1\tss 1}\non
\end{align}
such that $\la^{}_{n\tss i}=\la^{}_{i}$ for $i=1,\dots,n$,
the remaining
entries are all nonpositive
integers and the following inequalities hold:
\ben
\la'_{k\tss 1}\geqslant\la^{}_{k\tss 1}
\geqslant\la'_{k\tss 2}\geqslant
\la^{}_{k\tss 2}\geqslant \dots\geqslant
\la'_{k\ts k-1}\geqslant\la^{}_{k\ts k-1}\geqslant
\la'_{k\tss k}\geqslant\la^{}_{k\tss k}
\een
for $k=1,\dots,n$, and
\ben
\la'_{k\tss 1}\geqslant\la^{}_{k-1\ts 1}
\geqslant\la'_{k\tss 2}\geqslant
\la^{}_{k-1\ts 2}\geqslant \dots\geqslant
\la'_{k\ts k-1}\geqslant\la^{}_{k-1\ts k-1}
\geqslant\la'_{k\tss k}
\een
for $k=2,\dots,n$.

\begin{mthmb} {\em The vectors
\ben 
\theta^{}_{\La}=
\prod_{k=1,\dots,n}^{\longrightarrow}
\Bigg(F_{k,-k}^{{\ts}-\la'_{k\tss 1}}\prod_{i=1}^{k-1}
F_{k,-i}^{\la^{}_{k-1\ts i}-\la'_{k\ts i+1}}
F_{-i,-k}^{\la^{}_{k\tss i}-\la'_{k\ts i+1}}
\Bigg)\ts v_{\lambda}
\een 
parameterised by all  type {\sf C}
patterns $\La$
associated with
$\la$ form a basis of $V(\la)$.}
\end{mthmb}

The proof of
Theorem~B
is given in Section~\ref{sp1}, it is derived from  our general results on monomial bases
of multiplicity spaces. 
In Section~\ref{sp2}, we present another basis of $U(\lambda,\mu)$
with somewhat more complicated conditions
on the exponents of the monomials, which  can be extended inductively to a basis of $V(\lambda)$.
Furthermore, in Section~\ref{sec:wb}, we
produce a certain modified version $\zeta^{}_{\La}$
of the basis of $V(\la)$ constructed in \cite{m:br}
and derive explicit formulas for the action
of generators of the Lie algebra $\spa_{2n}$ in this basis. Then we demonstrate
in Section~\ref{sec:tm} that
the bases $\theta^{}_{\La}$ and $\zeta^{}_{\La}$ are related
via a triangular transition matrix. This also
gives another proof of Theorem B.

\medskip

{\it Acknowledgments.} We are grateful to Arkady Berenstein,   Xin Fang, Evgeny Feigin, Jacob Greenstein,  Peter Littelmann,  and Markus Reineke for useful discussions.
A part of this paper was written during the second author's visit to the University of Sydney.
She would like to thank the School of Mathematics and Statistics for warm hospitality and support.

\section{The FFLV approach to the branching problem} \label{fflv-br}

Let $\gt g$ be a complex reductive Lie algebra and $V(\lambda)$ be an irreducible  finite-dimensional
$\gt g$-module. Fix a   triangular decomposition
$\gt g=\gt n^-\oplus\gt h\oplus\gt n^+$.
The Lie algebra $\gt n^-$ has a standard basis consisting of root vectors $\{f_1,\ldots,f_N\}$.
We choose a total {\it monomial order} on the monomials $m\in\cS(\gt n^-)$ in this basis.
Recall that a monomial order is a {\it total} order satisfying
the following two conditions: 
\begin{itemize}
\item[$\diamond$] $1\le m$ for each monomial $m$,
\item[$\diamond$] if $m_1\le m_2$ and $m_3\le m_4$, then $m_1m_3\le m_2m_4$\,;
\end{itemize}
i.e., a monomial order is compatible with multiplication.
The order leads to a filtration on  $V(\lambda)$ as follows. 
Let $v_\lambda\in V(\lambda)$ denote a highest weight vector.
The enumeration of root vectors defines a sequence $f_1,f_2,\ldots,f_N$.    
Making use of this enumeration, or, equivalently, of this sequence,   
we identify $m=\prod\limits_{i=1}^N f_i^{\alpha_i}\in \cS(\gt n^-)$ with the product 
$f_1^{\alpha_1}\ldots f_N^{\alpha_N}\in {\U}(\gt n^-)$, which we denote by the same letter
$m$. 
In the expression $m v_\lambda$, the symbol $m$ stands for an element of $\U(\gt n^-)$. 
A monomial $\tilde m\in\cS(\gt g)$ is %
{\it essential} if $\tilde m v_\lambda$ does
not lie in the linear span of $\{m v_\lambda\}$ with $m<\tilde m$.
Let $\ess(V(\lambda))=\ess(\lambda)$ denote the set of essential monomials  related to $V(\lambda)$.
As was observed in \cite{ffl:fm}, $\{m v_\lambda\mid m\in\ess(\lambda)\}$ is a basis of
$V(\lambda)$ by the very construction. 

For any two finite-dimensional irreducible  $\gt g$-modules $V(\lambda)$ and $V(\lambda')$,
one has the  inclusion
\begin{equation}\label{+}
\ess(\lambda)\ess(\lambda')\subset\ess(\lambda+\lambda'),
\end{equation}
see
\cite[Prop.~2.11]{ffl:fm}.  
The proof of that proposition
works for any, not necessary homogeneous, monomial order. However, the authors remark in the proof that
they are using a homogeneous order and therefore can assume that the root vectors commute.
For completeness, we briefly outline the argument.

Suppose that $m=\prod\limits_{i=1}^N f_i^{a_i}$ is essential for $\lambda$ and
$m'=\prod\limits_{i=1}^N f_i^{a_i'}$ is essential for $\lambda'$.
Set $\tilde m=m m'$ in $\cS(\gt n^-)$. As an element of $\U(\gt n^-)$, the monomial
$\tilde m$ is equal to the product $f_N^{a_N+a_N'}\ldots f_1^{a_1+a_1'}$.  Let
$\bv=v_\lambda\otimes v_{\lambda'}$ be a highest weight vector of
$V_{\lambda+\lambda'}\subset  V_\lambda \otimes V_{\lambda'}$.
Then we have
$$
\tilde m \bv \in \prod_{i=1}^N \binom{a_i+a_i'}{a_i} m v_\lambda \otimes m' v_{\lambda'} \, + \,
\left(V_{\lambda}\otimes \left<\hat m v_{\lambda'} \mid \hat m<m'\right>_{\mK} \, + \,
\left<\hat m v_{\lambda} \mid \hat m<m\right>_{\mK} \otimes V_{\lambda'}\right).
$$
From this one can conclude that $\tilde m\in\ess(\lambda+\lambda')$.

The main novelty of our approach to the branching problem is that we combine the FFLV method with the more classical  theory of Zhelobenko. In particular, the {\it extremal projector} will be playing a
major role.

\subsection{The extremal projector} \label{sec-proj}
Let $\Delta^+$ be the set of positive roots of $\ggot$ which is
determined by the triangular decomposition so that $\gt n^+$ (resp., $\gt n^-$)
is spanned by the root vectors $e_{\alpha}$ (resp., $f_{\alpha}$) with $\alpha\in\Delta^+$.
Consider the $\gt{sl}_2$-triples
$\{f_{\alpha},h_\alpha,e_\alpha\}\subset \gt g$  and assume that the roots
are normalised to satisfy the condition $\al(h_{\al})=2$.
Set
\ben
p_\alpha = 1 + \sum\limits_{k=1}^{\infty} f_{\alpha}^k e_\alpha^k\ts
\frac{(-1)^k}{k! (h_\alpha+\rho(h_\alpha)+1)\ldots (h_\alpha+\rho(h_\alpha)+k)},
\een
$\rho$ is the half sum
of the positive roots.
This expression is regarded as an element of the algebra of formal series of monomials
\ben
f_{\alpha_1}^{r_1}\ldots f_{\alpha_N}^{r_N} e_{\alpha_N}^{k_N} \ldots e_{\alpha_1}^{k_1}
\quad \text{ with } \quad (k_1-r_1)\alpha_1+\ldots + (k_N-r_N)\alpha_N=0
\een
with coefficients in
the field of fractions of the commutative algebra $\U(\h)$.
Choose a numbering of positive roots, $\alpha_1,\ldots,\alpha_N$.
A total order on $\Delta^+$ is said to be  {\it normal}  if
either $\alpha < \alpha+\beta < \beta $ or $\beta < \alpha+\beta < \alpha$ for each pair of positive roots
$\alpha, \beta$ such that $\alpha+\beta\in\Delta$.
Choose a normal order $\alpha_1<\dots <\alpha_N$
and set
\ben
p=p_{\alpha_1}\ldots p_{\alpha_N}.
\een
The element $p$ is independent of the choice of a normal order and is known as
the {\em extremal projector}; see
Asherova, Smirnov, and Tolstoy~\cite{ast:po}, \cite{ast:dc}.
A more detailed description of its properties can be found
in the work by Zhelobenko~\cite{z:ep,z:it}. In particular,
$p$ is characterised by the properties $p^2=p$ and
\beql{eq-p}
e_\alpha p=p f_{\alpha}=0\qquad\text{for all}\qquad \al\in \De^+.
\eeq

\subsection{The specifics of branching} \label{sec-rules}
A subalgebra $\gt q\subset\gt g$ is a {\it reductive subalgebra} if
$\gt q$ is reductive and the centre of $\gt q$ consists of
$\ad_{\gt g}$-semisimple elements.

Let $\gt g_0\subset \gt g$ be a    reductive subalgebra normalised by $\gt h$.
Then $\gt g_0$ inherits the triangular decomposition,
$\gt g_0=\gt n_0^+\oplus\gt h_0\oplus\gt n_0^-$, where $\gt n_0^{\pm}=\gt n^{\pm}\cap\gt g_0$.
In order to see the branching rules $\gt g\downarrow \gt g_0$, we need a certain special monomial order.
Let $\gt n^-=\gt n_0^-\oplus \gt r$ be the $\gt h$-stable decomposition.
Write
$m=m_0m_1$, where $m_0\in\cS(\gt n_0^-)$ and $m_1\in\cS(\gt r)$.
Having two monomials $m=m_0m_1$ and $m'=m'_0m'_1$, we first compare $m_1$ with $m_1'$ and
if $m_1<m'_1$, then $m<m'$.
If $m_1=m'_1$, then we compare $m_0$ with $m_0'$.
The order on the $\cS(\gt n_0^-)$-factors is of no particular importance.
When identifying $m_0m_1\in\cS(\gt n^-)$ with an element of $\U(\gt n^-)$, we take
a monomial  from
$\U(\gt n_0^-)\U(\gt r)$.

Set $\U_+(\gt r):=\gt r\U(\gt r)$ and
let  $m_1\in \U_+(\gt r)$ be a monomial having our chosen sequence of factors.
The most crucial restriction on the monomial order is that
\begin{equation} \label{br-order}
xm_1v_\lambda= [x,m_1] v_\lambda\in\left<\tilde m v_\lambda \mid  \tilde m\in\cS(\gt n^-), \tilde m< m_1\right>_{\mK}
\end{equation}
for each dominant weight $\lambda$ and each $x\in\gt n_0^+$. We will assume that it is satisfied.
If $\tilde m<m_1$ and $m_1\in\cS(\gt r)$,
then $\tilde m=\tilde m_0\tilde m_1$, where $\tilde m_1<m_1$. Therefore
\eqref{br-order} implies that
\begin{equation} \label{br-order-2}
Xm_1v_\lambda\in\left<\tilde m v_\lambda \mid  \tilde m\in\cS(\gt n^-), \tilde m< m_1\right>_{\mK}
\end{equation}
for each dominant weight $\lambda$ and each $X\in\U(\gt g_0)\gt n_0^+$.

\begin{lm} \label{lem-3}
Suppose that $[\gt n_0^+\oplus\gt r,\gt r]\subset\gt r$. Then there is a
 natural way to guarantee that \eqref{br-order} is satisfied. Namely, one has to compare the $\gt h$-weights $\nu,\nu'$ of 
$m_1, m_1'\in\cS(\gt  r)$ first and say that if $\nu-\nu'\in\Delta^+$, then $m_1<m_1'$. 
\end{lm}
\begin{proof}
 We may assume that $x\in\gt n_0^+$ is 
a root vector corresponding to a positive root $\beta$ of $\gt g$.  In this case the weight of $xm_1$ is $\nu+\beta$. 
By the assumptions on $\gt r$, 
$xm_1v_\lambda \in\left<\tilde m v_\lambda \mid  \tilde m\in\cS(\gt r)\right>_{\mK}$, where the weight of each $\tilde m$ equals $\tilde\nu=\nu+\beta$. Hence here $\tilde m<m_1$ as required.
\end{proof}

Let $p$ be the extremal projector associated with $\gt g_0$. Set $N'=\dim\gt n_0^-$.
Suppose that $w\in V(\lambda)$ is a weight vector such that $pw$ is well-defined.
Then  $pw$ is equal to $w$ plus a finite linear combination of  expressions
$$
f_{\alpha_1}^{r_1}\ldots f_{\alpha_{N'}}^{r_{N'}} e_{\alpha_{N'}}^{k_{N'}} \ldots e_{\alpha_1}^{k_1} w,
$$
where $k_1+\ldots+k_{N'} >0$.
By \eqref{br-order-2}, we have 
\begin{equation}\label{eq-p2}
p m_1v_\lambda\in m_1 v_\lambda + \left<\tilde m v_\lambda \mid \tilde m<m_1\right>_{\mK}
\end{equation}
whenever $pm_1v_\lambda$ is well-defined (that is, the values of the denominators
occurring in $pm_1v_\lambda$ are not zero).

Recall that $V(\la)^+_{\mu}=(V(\lambda)^{\gt n_0^+})_{\mu}$ stands for the subpace
of $\gt g_0$-highest weight vectors in $V(\la)$ of $\gt h_0$-weight $\mu$.

\begin{prop} \label{lm-ess-br}
Keep the above notation and the assumptions on the monomial order. 
Then $pm_1v_\lambda$ is well-defined for each $m_1\in\ess(\lambda)\cap \cS(\gt r)$ and
the set of vectors
\ben
\{ pm_1 v_\lambda \mid m_1\in\ess(\lambda)\cap \cS(\gt r)\}
\een
is a basis of the
subspace $V(\lambda)^{+}=V(\lambda)^{\gt n_0^+}=\bigoplus\limits_{\mu} V(\lambda)_{\mu}^+$.
\end{prop}
\begin{proof}
One  observes easily that $V(\lambda)^{+}$ is spanned by
$p m v_\lambda$, where $m\in\ess(\lambda)$ and the $\gt h_0$-weight of $m v_\lambda$ is dominant for $\gt g_0$.
If $m\not\in \cS(\gt r)$, then $pm=0$ by \eqref{eq-p} and our assumption on the sequence of factors in $\U(\gt n^-)$.
It remains to prove that the vectors in question 
are well-defined and linearly independent.

Assume that  $pm_1v_\lambda$ is not well-defined for some
$m_1\in \ess(\lambda)\cap \cS(\gt r)$. Then the weight of $u=m_1 v_\lambda$ is not dominant for $\gt g_0$.
Let $\{e,h,f\}\subset\gt g_0$ be an $\gt{sl}_2$-triple such that $e\in\gt n_0^+$ is a simple root vector and
$hu=-d u$ for some $d>0$.
By the standard $\gt{sl}_2$-theory, which includes classification of the finite-dimensional
$\gt{sl}_2$-modules, there is $k=d+2k'$ such that $u$ lies in $\bigoplus\limits_{t=d}^k \cS^t \mK^2$
up to an isomorphism.
Therefore one can find elements $a(t)\in \mK$ such that
$$u=\sum\limits_{t=d}^{d+k'} a(t)  f^t e^t u.$$ 
Here each $e^t m_1 v_\lambda$\,, and hence also each $f^t e^t m_1 v_\lambda$, lies in
$\left<\tilde m v_\lambda \mid \tilde m<m_1\right>_{\mK}$\,, see \eqref{br-order-2}.  Therefore $m_1$ is not essential for
$\lambda$.

Assume finally that a non-trivial linear combination of
$p m_1 v_\lambda$ with  $m_1\in\ess(\lambda)\cap \cS(\gt r)$ is equal to zero.
Then by \eqref{eq-p2},  the largest monomial appearing  in it with a non-zero coefficient is
not essential for $\lambda$.
\end{proof}

The inclusion \eqref{+} justifies the following definition.

\begin{df} \label{br-group}
The subset $$\Gamma=\Gamma_{\gt g\downarrow \gt g_0}:=\{(\lambda,m_1) \mid m_1\in\ess(\lambda)\cap \cS(\gt r)\}\subset\gt h^*\times\cS(\gt r), \ \ \text{where $\lambda$ is dominant},$$
is called the {\it branching semigroup} of $\gt g\downarrow \gt g_0$.
Set also $\Gamma(\lambda)=\{m_1\mid (\lambda,m_1)\in\Gamma\}$.
\end{df}

Note that the above objects depend on the basis of $\gt n^-$,  on the monomial order, and on the
sequence of factors in $\U(\gt n^-)$.
A standard procedure for calculating  $\Gamma$ is to consider first
small values of $\lambda$, like the {\it fundamental weights} $\varpi_i$, obtain enough elements
in $\Gamma(\lambda)$, and then compare the cardinality with the dimension of
$V(\lambda)^{+}$. However, this approach can produce a description of $\Gamma$ only if the semigroup is finitely generated.
It is conjectured in \cite{xfp} that $\Gamma$ is always finitely generated  in our context as well as in
a less restrictive one considered  there. Partial positive results in this direction are obtained in
\cite[Sect.~12]{xfp}.

\begin{ex} \label{ex-sl}
As we will see below, the semigroup $\Gamma=\Gamma_{\gt{sl}_n\downarrow\gt{gl}_{n-1}}$ is generated by
the pairs $(\varpi_i, m_1)$ with
$m_1\in \Gamma(\varpi_i)$ and $1\le i<n$.
\end{ex}

\subsection{Inductive bases for $V(\lambda)$} \label{sec-V}
Next we show how
branching rules lead to constructions of FFLV-type bases. 
\begin{prop}\label{other}
We have $m_0m_1\in \ess(\lambda)$ if and only if
$m_1\in\Gamma_{\gt g\downarrow \gt g_0}(\lambda)$ and $m_0\in\ess(\mu)$, where
$\mu=\mu(m_1 v_\lambda)$ is the weight of $m_1 v_\lambda$ w.r.t. $\gt h_0$.
\end{prop}
\begin{proof}
Suppose first that $m_0m_1\in \ess(\lambda)$.
If $m_1$ is not essential for $\lambda$,
then $$m_1 v_\lambda=\sum\limits_{k} A(k) a_0(k) a_1(k) v_\lambda$$
for some $A(k)\in\mK$, some monomials $a_0(k)\in\U(\gt n_0^-)$  and
$a_1(k)\in \U(\gt r)$, and  $a_1(k)<m_1$ for all $k$. In this case $m_0 a_0(k) a_1(k) < m_0 m_1$ for each $k$ and
hence $m_0m_1$ is not essential, a contradiction.

If $m_0\not\in\ess(\mu)$, then
$$
m_0 pm_1 v_\lambda = \sum\limits_k B(k) b_0(k) p m_1 v_\lambda
$$
for some $B(k)\in\mK$, some monomials $b_0(k)\in\U(\gt n_0^-)$, and we have
$b_0(k)<m_0$ for each $k$. Since $m_1 v_\lambda$ is the leading term of
$pm_1 v_\lambda$ by \eqref{eq-p2}, we conclude  that $m_0m_1$ is not essential,  a contradiction.

Now we know that
$$
|\ess(\lambda)| \le \sum_{m_1\in \Gamma(\lambda)} |\ess(\mu(m_1 v_\lambda))| = \dim V(\lambda).
$$
Since also $|\ess(\lambda)|=   \dim V(\lambda)$, we can conclude that each product
$m_0 m_1$, where $m_1$ and $m_0$ are essential for $\lambda$ and $\mu$,
respectively, is essential for $\lambda$. This completes the proof.
\end{proof}

\begin{rmk}  One can also give  a direct proof for  the inclusion
$\ess(\mu)\Gamma_{\gt g\downarrow \gt g_0}(\lambda)\subset\ess(\lambda)$ avoiding dimension reasons.
\end{rmk}

\subsection{The Gelfand--Tsetlin order in type {\sf A}}  \label{sec-A-proofs}
Here we show how effortlessly the FFLV method leads to a construction of
the basis described in Theorem~A.

Take $\gt g=\gt{gl}_n$ and $\gt g_0=\gt{gl}_{n{-}1}$ that is the span of $E_{ij}$ with
$1\leqslant i,j< n$.
Then $\gt r$
is the linear span of $E_{n\,k}$ with $1\le k < n$. Note that $[\gt r,\gt r]=0$. Hence
the sequence of factors in $m_1\in\U(\gt n^-)$ is of no significance.
The $\gt h_0$-weights of $E_{n\,k}$ with $1\le k < n$ are linearly independent. If
$m_1\ne \tilde m_1$ and $p m_1 v_\lambda\ne 0$, then $p m_1 v_\lambda\ne p \tilde m_1 v_\lambda$.
The branching $\gt{gl}_n \downarrow \gt{gl}_{n-1}$ is {\it multiplicity free}, which is the key point
of  \cite{gt:fdu}.
Given a highest weight $\mu$ such that $U(\lambda,\mu)=\Hom_{\gt g_0}(V'(\mu),V(\la))\ne 0$, there is a unique way to
write the corresponding $m_1\in\ess(\lambda)$, which exists  by Proposition~\ref{lm-ess-br}. Since the branching rules are well-known,  the  description of  $\Gamma(\lambda)$ results from Proposition~\ref{lm-ess-br} immediately.
Write $\lambda=(\lambda_1,\ldots,\lambda_n)$ with $\lambda_k-\lambda_{k+1}\in\mathbb Z_{+}$
for $k=1,\dots,n-1$.

\begin{cl} \label{cl-gl}
For each monomial order satisfying the assumptions of Section~\ref{sec-rules},
 $$
 \Gamma_{\gt{gl}_n\downarrow \gt{gl}_{n{-}1}}(\lambda)=\{ E_{n\,1}^{\alpha_1}\ldots E_{n\,n-1}^{\alpha_{n-1}}
  \mid \alpha_{k} \le \lambda_k-\lambda_{k{+}1} \}.
$$
Hence, the semigroup
$\Gamma_{\gt{gl}_n\downarrow \gt{gl}_{n{-}1}}$ is generated by
the sets
$\{(\varpi_k,1), (\varpi_k,E_{n\,k})\}$ with $1\le k<n$ 
together with the $1$-dimensional representations of $\gt{gl}_n$.
\end{cl}

The central elements of $\gt{gl}_n$ act on $L(\lambda)$ as scalars and any
$1$-dimensional representation of $\gt{sl}_n$ is trivial.  Thereby the statement of
Example~\ref{ex-sl} follows from Corollary~\ref{cl-gl}. For the sake of briefness, one says also  that
$\Gamma_{\gt{sl}_n\downarrow \gt{gl}_{n{-}1}}$ is generated by the fundamental weights or by
$\Gamma_{\gt{sl}_n\downarrow \gt{gl}_{n{-}1}}(\varpi_i)$.

An example of a suitable, i.e., satisfying \eqref{br-order},   monomial order on $\cS(\gt r)$ is
the lexicographical order
on $E_{n\,1}^{\alpha_1}\ldots E_{n\,n-1}^{\alpha_{n-1}}$,
which is also the {\it right} lexicographical order on the tuples $(\alpha_{n-1}, \ldots , \alpha_1)$. 

The elements of $\Gamma_{\gt{gl}_n\downarrow \gt{gl}_{n{-}1}}(\lambda)$
can be parameterised by  the Gelfand--Tsetlin patterns
$\Lambda$, as defined in \eqref{aconl}.
Each  such $\Lambda$ corresponds to the monomial
  $$
m_1(\Lambda)= E_{n\,1}^{\lambda_{n\,1}-\lambda_{n-1\, 1}} \ldots E_{n\,n-1}^{\lambda_{n\,n-1}-\lambda_{n-1\,n-1}}.
$$

Arguing inductively with the use of Proposition~\ref{other},
we restrict $L(\lambda)$  further to $\gt{gl}_{n{-}2}$, $\gt{gl}_{n{-}3}$, and so on.
Taking
the sequence of factors
$$
E_{2\,1}^{\alpha_{2,1}}  E_{3\,1}^{\alpha_{3,1}}  E_{3\,2}^{\alpha_{3,2}}
\ldots E_{n\,1}^{\alpha_{n,1}}\ldots E_{n\,n-1}^{\alpha_{n,n-1}}
$$
in $\U(\gt g)$
and the   lexicographical order at each step
we obtain the basis of Theorem~A.
An alternative   way to express this basis is to write
$$
\ess(\lambda)=\left\{\prod E_{i\,j}^{\alpha_{i,j}}    \mid
\alpha_{i,j}\le \lambda_{j}-\lambda_{j+1}+\sum\limits_{k=i+1}^{n}(\alpha_{k,j+1} -\alpha_{k,j})
 \right\}.
$$
This is the set of inequalities given in \cite[Introduction]{Markus-dis}.
The same inequalities are used in \cite[Sect.~6]{bz-duke} for a description of a
different, but related, basis.

The inductive argument shows also that
the semigroup $\Gamma=\Gamma_{\gt{sl}_n\downarrow \{0\}}$ is generated by
$\Gamma(\varpi_k)$ with $1\le k<n$.

The next example is crucial  for the symplectic case.

\begin{ex}\label{ex-middle}
Consider  $\gt{gl}_{n{-}1}\subset\gt{gl}_{n{+}1}$ embedded as the middle $(n-1)\times(n-1)$-square.
For elements of $\U(\gt r)$, we are using
the following sequence of root vectors: $$\prod\limits_{k=2}^{n}  E_{n{+}1\,k}^{\alpha_{n+1,k}}
 \prod\limits_{k=2}^{n{+}1} E_{k\,1}^{\alpha_{k,1}}.$$
 The monomial order is given by the right lexicographical order on the tuples
 $$(\alpha_{n+1,n},\ldots,\alpha_{n+1,2},\alpha_{2,1},\ldots,\alpha_{n+1,1}).$$
Here
$\Gamma_{\gt{gl}_{n{+}1}\downarrow \gt{gl}_{n{-}1}}(\lambda)=\Gamma_{\gt{sl}_{n{+}1}\downarrow \gt{gl}_{n{-}1}}(\lambda)$ is equal to
$$
\left\{ \prod\limits_{k=2}^{n}  E_{n{+}1\,k}^{\alpha_{n+1,k}}
 \prod\limits_{k=2}^{n{+}1} E_{k\,1}^{\alpha_{k,1}}
  \mid \alpha_{k{+}1,1} \le \lambda_k-\lambda_{k+1}\fand
  \alpha_{n{+}1,k}\le \lambda_k-\lambda_{k+1} +\alpha_{k,1} - \alpha_{k{+}1,1} \right\}.
$$
The branching semigroup $\Gamma_{\gt{gl}_{n{+}1}\downarrow \gt{gl}_{n{-}1}}$ is generated by  1-dimensional representations of $\gt{gl}_{n{+}1}$ and by  the essential monomials of the fundamental weights.
Record that
\begin{equation} \label{A-middle}
\begin{array}{l}
\Gamma_{\gt{gl}_{n{+}1}\downarrow \gt{gl}_{n{-}1}}(\varpi_1)=\{1,E_{2\,1}, E_{n{+}1\,2} E_{2\,1}\}; \\
\Gamma_{\gt{gl}_{n{+}1}\downarrow \gt{gl}_{n{-}1}}(\varpi_k)=\{1,E_{k{+}1\,1}, E_{n{+}1\,k},  E_{n{+}1\,k{+}1} E_{k{+}1\,1}\} \ \ \text{if} \ \ 2\le k < n; \\
\Gamma_{\gt{gl}_{n{+}1}\downarrow \gt{gl}_{n{-}1}}(\varpi_n)=\{1,E_{n{+}1\,1}, E_{n{+}1\,n} \}. \\
 \end{array}
\end{equation}
If we replace $\gt{gl}_{n+1}$ with $\gt{sl}_{n+1}$, then the 
1-dimensional representations disappear from the generating set.  
\end{ex}

\section{Symplectic branching rules}\label{br}

In this section we take $\gt g=\gt{sp}_{2n}$ and use the presentation
of the symplectic Lie algebra defined in the Introduction. The subalgebra
$\gt g_0=\gt{sp}_{2n{-}2}$ is spanned by the elements $F_{ij}$ with
$-n+1\leqslant i,j\leqslant n-1$.
Let $\gt g=\gt n^-\oplus\gt h\oplus\gt n^+$ be the triangular decomposition,
where $\h$ is the Cartan subalgebra of $\ggot$ with the basis
$\{F_{11},\dots,F_{nn}\}$, while the subalgebra $\gt n^+$ (resp., $\gt n^-$)
is spanned by the elements $F_{ij}$ with $i<j$ (resp., $i>j$).
We have a vector space decomposition $\gt n^-=\gt n_0^-\oplus\gt r$, where
$\gt r=\left<F_{n\,i} \mid i<n\right>_{\mK}$ is a Heisenberg Lie
algebra and $[\gt r,\gt r]$ is spanned by $F_{n,-n}$.
The elements from
different pairs $(F_{n,i},F_{i,-n})$, $(F_{n,j},F_{j,-n})$ commute with each other and
 $[F_{n,i},F_{i,-n}]=F_{n,-n}$,  
where $F_{n,-n}$ is a central element of $\gt r$. Note that $[\gt g_0,\gt r]\subset\gt r$.

\subsection{The Gelfand--Tsetlin-type order in the symplectic case} \label{sp1}
We will describe a rather elaborate monomial order on $\cS(\gt r)$ suggested by the
structure of the branching semigroup
of Example~\ref{ex-middle}.

\begin{df} \label{m-order2}
Define a monomial order on $\cS(\gt r)$ by the following rule. The monomial
\begin{equation}\label{sp-seq1}
F_{n,-n}^{\alpha_{1}} F_{n,-1}^{\alpha_2}F_{n,-2}^{\alpha_3} \ldots F_{n,-n+1}^{\alpha_n} F_{n,1}^{\alpha_{n{+}1}} \ldots
F_{n,n{-}1}^{\alpha_{2n{-}1}}
\end{equation}
given  by $\bar\alpha=(\alpha_1,\ldots,\alpha_{2n-1})$ is smaller than the monomial
given by  $\bar\alpha'=(\alpha_1',\ldots,\alpha_{2n-1}')$ if and only if
either $\nu-\nu'\in\Delta^+$ for the $\gt h$-weights $\nu,\nu'$ of these monomials or 
$\nu-\nu'\not\in\Delta$ and $\bar\alpha<\bar\alpha'$ in the lexicographical order. 
\end{df}

\begin{lm} \label{3drule}
Choose the sequence of factors in $\U(\gt r)$ as in \eqref{m-order2}. 
Then the  monomial order
of Definition~\ref{m-order2} satisfies \eqref{br-order}.
\end{lm}
\begin{proof}
Since $[\gt g_0\oplus\gt r,\gt r]\subset\gt r$, the statement follows from Lemma~\ref{lem-3}. 
\end{proof}

Let $\widetilde{\Gamma}$ be the branching semigroup of $\gt g\downarrow\gt g_0$
defined by the  sequence of root vectors as in \eqref{m-order2} and the monomial order of
Definition~\ref{m-order2}.

\begin{thm} \label{GT-semi}
The semigroup $\widetilde{\Gamma}$ is generated by the
pairs $(\varpi_i, m_1)$, where $\varpi_i$ is a fundamental weight and $m_1\in\widetilde\Gamma(\varpi_i)$.
Under a suitable identification, $\widetilde{\Gamma}$  is defined by the same inequalities as the semigroup  $\Gamma_{\gt{sl}_{n{+}1}\downarrow \gt{gl}_{n{-}1}}$ described in Example~\ref{ex-middle}. 
\end{thm}
\begin{proof}
We use the bijection
between the sets
$$
\{F_{n,k} \mid -n\le k < n,\  k\ne 0\}\fand
\{E_{n{+}1\,k}, E_{t\,1} \mid  1\le k \le n,\ \ 2\le t\le n\}
$$
which takes
$F_{n,-n}$ to $E_{n{+}1\,1}$, the vector $F_{n,-k}$ with $1\le k < n$ to
$E_{n{+}1\,n-k+1}$, and $F_{n,k}$ to $E_{n{+}1{-}k\,1}$.
Using the same letters, $\varpi_i$, for the fundamental weights of both  $\gt{sp}_{2n}$ and $\gt{sl}_{n+1}$,
we  identify also the highest weights $\lambda=\sum c_i\varpi_i$ of $\gt{sp}_{2n}$ and $\gt{sl}_{n{+}1}$.
Then the standard branching theory assures that $|\widetilde{\Gamma}(\lambda)|=
|\Gamma_{\gt{sl}_{n{+}1}\downarrow \gt{gl}_{n{-}1}}(\lambda)|$, see e.g. \cite{m:gtb} and patterns in the Introduction. 
Since we have the property $\widetilde{\Gamma}(\lambda)\widetilde{\Gamma}(\lambda')\subset \widetilde{\Gamma}(\lambda+\lambda')$, see \eqref{+},
it remains to show that the image of each  $\widetilde{\Gamma}(\varpi_k)$ is exactly
$\Gamma_{\gt{sl}_{n{+}1}\downarrow \gt{gl}_{n{-}1}}(\varpi_k)$. The latter is presented in \eqref{A-middle}.
Let $\bv_r\in V(\varpi_r)$ be a highest weight vector.

Take $\varpi_1$. Here $|\widetilde{\Gamma}(\varpi_1)|=3$.
Notice that $F_{n,n{-}1} \bv_1\ne 0$ is a highest weight vector of $\gt g_0$ and that
$F_{n,n{-}1}$ is the smallest root vector in the monomial order. Therefore  $F_{n,n{-}1}\in\ess(\varpi_1)$.
This root vector is mapped
to $E_{2\,1}\in \Gamma_{\gt{sl}_{n{+}1}\downarrow \gt{gl}_{n{-}1}}(\varpi_1)$.
It remains to take care of  the second copy of the trivial representation, which
one obtains by applying either $F_{n,-n}$ or  $F_{k,-n} F_{n,k}$ with $1\le k < n$ to $\bv_1$.
The smallest monomial here is $F_{n-1,-n}F_{n,n-1}$. Since $F_{n-1,-n}$ is mapped to $E_{n{+}1\,2}$, we see that the image of $\widetilde{\Gamma}(\varpi_1)$ is exactly $\Gamma_{\gt{sl}_{n{+}1}\downarrow \gt{gl}_{n{-}1}}(\varpi_1)$.

Take next $\varpi_k$ with $2\le k <n$. Here $|\widetilde{\Gamma}(\varpi_k)|=4$. The
root vectors $F_{k-n,-n}$ and $F_{n,k-n-1}$ are essential for $\varpi_k$.
The root vector  $F_{n,-n}$ is not essential,  because it can be replaced by
$F_{n,k-n}F_{k-n,-n}$, which is smaller. We have also $F_{n,i}\bv_k=0$ if $n-k<i\le n-1$.
Therefore, 
it remains
to choose the smallest monomial among $F_{n,t}F_{t,-n}$ with $k-n\le t\le -1$.
This is exactly
$F_{n,k-n}F_{k-n,-n}$. Thus the image of $\widetilde{\Gamma}(\varpi_k)$ is $\Gamma_{\gt{sl}_{n{+}1}\downarrow \gt{gl}_{n{-}1}}(\varpi_k)$.

Finally take $\varpi_n$, where we have $\widetilde{\Gamma}(\varpi_n)=\{1,F_{n,-n}, F_{1,-n}\}$.
Note that $F_{n,-n}$ is mapped to $E_{n{+}1\,1}$ and $F_{1,-n}$ to $E_{n{+}1\,n}$.
This finishes the proof.
\end{proof}

If a dominant weight $\lambda=\sum\limits_{k=1}^n c_k \varpi_k$  of $\gt{sp}_{2n}$ 
is presented by a tuple
$(\lambda_1,\lambda_2,\ldots,\lambda_n)$ with $0\ge \lambda_1\ge \ldots\ge \lambda_1$  
as in the Introduction, cf. \eqref{hwrepbcd},
then $c_1=\lambda_{n-1}-\lambda_n$, for
$2\le k<n$,  we have $c_k=\lambda_{n-k}-\lambda_{n-k-1}$, and $c_n=-\lambda_1$.
Consistently, we write  $\mu=(\mu_1,\ldots,\mu_{n-1})$ with $0\ge \mu_1\ge\ldots \ge \mu_{n-1}$.
Taking this into account and using bijections between the branching semigroups and
the corresponding patterns
(Gelfand--Tsetlin patterns and type {\sf C} patterns),
we obtain the following statement.

\begin{cl}\label{basis}
The vector space $V(\lambda)_{\mu}^+$ has a basis 
$$
\left\{ \, p F_{n,-n}^{-\nu_1} F_{n, -n+1}^{ \mu_{n-1}{-}\nu_n} F_{n,n-1}^{\lambda_{n-1}{-}\nu_n} \ldots
F_{n, -k}^{ \mu_k{-}\nu_{k+1}} F_{n,k}^{\lambda_k{-}\nu_{k+1}} \ldots F_{n, -1}^{ \mu_1{-}\nu_2} F_{n,1}^{\lambda_1{-}\nu_2} v_\lambda\,\right\},
$$
parameterised by the $n$-tuples $\nu=(\nu_1,\dots,\nu_n)$ satisfying the
betweenness conditions
\begin{equation} \label{eq-nu}
\bal
&0\geqslant\nu_1\geqslant\la_1\geqslant\nu_2\geqslant
\la_2\geqslant \cdots\geqslant
\nu_{n-1}\geqslant\la_{n-1}\geqslant\nu_n\geqslant\la_n,\\
&0\geqslant\nu_1\geqslant\mu_1\geqslant\nu_2
\geqslant\mu_2\geqslant \cdots\geqslant
\nu_{n-1}\geqslant\mu_{n-1}\geqslant\nu_n.
\eal
\end{equation}
\end{cl}

Going  inductively through the chain of subalgebras
\begin{equation} \label{sp-chain}
\gt{sp}_2\subset \ldots \subset\gt{sp}_{2n-2}\subset \gt{sp}_{2n}
\end{equation}
and using Proposition~\ref{other} at each step, we obtain the basis of Theorem~B.
The chain
defines also the branching
semigroup  $\tilde\Gamma_{\gt{sp}_{2n}\downarrow \{0\}}$,  where
the order of Definition~\ref{m-order2} and the sequence of factors \eqref{sp-seq1} are used at each step.

\begin{rmk}
Arguing inductively, one can show that $\tilde\Gamma_{\gt{sp}_{2n}\downarrow \{0\}}$ is generated
by $\tilde\Gamma_{\gt{sp}_{2n}\downarrow \{0\}}(\varpi_k)$ with $1\le k\le n$.
This implies that
$\tilde\Gamma_{\gt{sp}_{2n}\downarrow \{0\}}$ is  {\it saturated}, i.e.,
$\tilde\Gamma_{\gt{sp}_{2n}\downarrow \{0\}}(N\lambda)= (\tilde\Gamma_{\gt{sp}_{2n}\downarrow \{0\}}(\lambda))^N$  for any  $N\in\mathbb N$ and any dominant weight $\lambda$.
In this situation, there is a nice toric degeneration of the complete flag variety
in the spirit of  \cite[Sect.~15]{xfp} and \cite[Sect.~10]{ffl:fm}.
\end{rmk}

\subsection{A different, more natural, order} \label{sp2}
In this section, it is more convenient to  use different indices for the  matrix realisation of $\gt g=\gt{sp}_{2n}$.
Now $\gt g$
is the linear span of   $F_{i\,j}$ with $i,j\in\{1,\dots,2n\}$, where
\begin{equation}
 F_{i\, j}
=E_{i\, j}-\ve_i\ts\ve_j\ts E_{j'\, i'}, \qquad i'=2n-i+1,
\end{equation}
$\ve_i=1$ for $i\le n$, and
$\ve_i=-1$ for $i> n$.
The subalgebra $\ggot_0=\spa_{2n-2}$ is spanned by the elements
$F_{i\,j}$ with $i,j\in\{2,\dots,2n-1\}$.
We have $\gt r=\left<F_{2n\,k} \mid 1\le k <2n\right>_{\mK}$.

This alternative presentation of $\gt g$ requires a change in the  convention  for
tuples $\lambda=(\lambda_1,\ldots,\lambda_n)$. Now $\lambda_1\ge \lambda_2\ge \ldots\ge \lambda_n\ge 0$
unlike the Introduction.
Fix highest weights
$\lambda=(\lambda_1,\ldots,\lambda_n)$ for $\ggot$ and $\mu=(\mu_2,\ldots,\mu_n)$ for $\ggot_0$,
where we suppose also that 
$\mu_2\ge \mu_3\ge \ldots\ge \mu_n\ge 0$. Assume that
the multiplicity space $U(\lambda,\mu)$ is nonzero.

Set $a_i=|\lambda_i-\mu_i|$ for $i\ge 2$ and
define the monomial $Y(\mu)=y_n^{a_n}\ldots y_2^{a_2}$ by the rule:
$$y_i=F_{i\,1} \ \text{ if }  \ \lambda_i\le \mu_i\fand
y_i=F_{2n\,i} \  \text{ if } \  \lambda_i>\mu_i.$$

Now use a non-zero vector  $\xi_\mu\in V(\lambda)_\mu^{+}$ defined in
 formula \cite[(9.69)]{m:yc}, cf.~\eqref{bcdhv}. We need the existence of this vector and the computation of its weight
w.r.t. to $F_{1\,1}$. In the notation of this section,   formula \cite[(9.69)]{m:yc} leads to the following 
$$
F_{1\,1} \xi_\mu=\Big(\lambda_1  - \sum\limits_{i=2}^n (2\max(\lambda_i,\mu_i)-\lambda_i-\mu_i)
\Big)\xi_\mu= \Big(\lambda_1-\sum\limits_{i=2}^n  a_i\Big) \xi_\mu.
$$
Hence, the $\h$-weight of $\xi_\mu$ coincides with that of
the vector $Y(\mu)\tss v_\lambda$.

\begin{rmk}\label{xi-mu}
If $m_1\in\cS(\gt r)$ is an eigenvector of $\gt h$ of the same weight as $Y(\mu)$, then
$m_1$ lies in $\mathbb C Y(\mu)$.
Thus
$\xi_\mu = p Y(\mu) v_\lambda$ and also
$\xi_\mu = \prod\limits_{i=2}^{n} (p y_i)^{a_i}\tss  v_\lambda$, up to non-zero scalar factors.
\end{rmk}

We would like to
find inequalities for $b,b_n,\ldots,b_2$ such that the corresponding
vectors
$$
p F_{2n\,1}^b (F_{2n\,n}F_{n\,1})^{b_n}\ldots(F_{2n\,2}F_{2\,1})^{b_2} Y(\mu) v_\lambda
$$
form a basis of $V(\lambda)_{\mu}^+$. 
For this purpose, 
the most natural monomial order on $\cS(\gt r)$ is suitable.

For a vector $\bar\alpha= (\alpha_2,\ldots,\alpha_{2n-1})$, set $|\bar\alpha|=\sum\limits_{k=2}^{2n-1}\alpha_k$.

\begin{df} \label{m-order}
We say that $F_{2n\,1}^{\alpha_{2n}} F_{2n\,2}^{\alpha_2} \ldots F_{2n\,2n{-}1}^{\alpha_{2n{-}1}} <
F_{2n\,1}^{\beta_{2n}} F_{2n\,2}^{\beta_2} \ldots F_{2n\,2n{-}1}^{\beta_{2n{-}1}}$ if and only if \\[.2ex]
either  $|\bar\alpha|<|\bar\beta|$ or $|\bar\alpha|=|\bar\beta|$  and there is $k$ such that $2\le k\le 2n$ and
$$
\alpha_k<\beta_k, \ \ \alpha_i=\beta_i \ \ \text{for all} \ \ i<k.
$$
\end{df}

A few remarks on the definition are due. \\[.2ex]
(1) Since we are comparing the degrees first, the sequence of factors of $m_1\in\U(\gt r)$ is not significant for
being essential.  \\[.2ex]
(2) Independently of the sequence of factors in $\U(\gt r)$, the chosen order satisfies \eqref{br-order}.
Therefore, by Proposition~\ref{lm-ess-br}, the subspace $V(\lambda)^{+}$ has a basis
$\{ pm_1 v_\lambda \mid m_1\in\ess(\lambda)\cap \cS(\gt r)\}$.

\begin{lm}\label{lm-gamma}
We have
{\begin{align}
& \Gamma(\varpi_1)=\{1,F_{2n\,1},F_{2\,1}\},  &  \nonumber \\
 & \Gamma(\varpi_k)=\{1,F_{2n\,1},F_{2n\,k},F_{k{+}1\,1}\} \ \ \text{if} \ \  2\le k<n,  \nonumber \\
& \Gamma(\varpi_n)=\{1,F_{2n\,1},F_{2n\,n}\}. \nonumber
\end{align} }
\end{lm}
\begin{proof}
The statements can be obtained by direct calculations.
\end{proof}

The dimension  of $U(\lambda,\mu)$ is the product of
$n$ positive integers $(d_1+1)\dots(d_n+1)$, where
$$
d_i=\min(\lambda_i,\mu_i)-\max(\lambda_{i+1},\mu_{i+1}),
$$
assuming that $\min(\lambda_1,\mu_1)=\la_1$ and $\lambda_{n+1}=\mu_{n+1}=0$; see e.g. \cite{m:gtb}.

Consider the $\gt{sl}_2$-triple $\{\frac{1}{2}F_{2n\,1}, F_{1\,1}, \frac{1}{2}F_{1\,2n}\}$.
The  subalgebra of $\gt g$ spanned by this triple acts on $U(\lambda,\mu)$ as on 
$\cS^{d_1}\mK^2\otimes\ldots \otimes \cS^{d_n}\mK^2$. Moreover,
$$\xi_\mu\in V(\lambda)_{\mu}^+\cong U(\lambda,\mu)$$ is a highest weight vector of this
representation
and its $F_{1\,1}$-weight is equal to $d_1+\ldots+d_n$.

For a vector $Y=y_n^{a_n}\ldots y_2^{a_2}$, where each $y_i$ is either $F_{i\,1}$ or
$F_{2n\,i}$ and $a_i\in\mathbb Z_{\ge 0}$ are arbitrary, set
 $$
\iota_{i}=\left\{ \begin{array}{l} 0 \ \text{ if } y_i=F_{i\,1}, \\
  1 \ \text{ if } y_i=F_{2n\,i}. \end{array}\right.
$$
This defines a vector $\bar\iota=(\iota_2,\ldots,\iota_{n})$, which depends on $Y$. Set $\iota_1=0$ and $a_{n+1}=0$.

We have
$$
\lambda=(\lambda_1-\lambda_2)\varpi_1 + \ldots + (\lambda_{n{-}1}-\lambda_n)\varpi_{n{-}1} +\lambda_n \varpi_n.
$$
Set $c_n=\lambda_n$ and $c_k=(\lambda_k-\lambda_{k{+}1})$ for $k<n$.
Suppose that $\xi_\mu = pYv_\lambda\ne 0$ for some $Y$ as above.
It is not difficult to see  then that $Y=Y(\mu)$ and
\begin{equation} \label{iota}
d_k = c_k - \iota_k a_k  - (1-\iota_{k+1}) a_{k+1}
\end{equation}
for each $k\ge 2$.
Informally speaking,   each $y_i$ in $Y$ decreases
$c_k$ by $1$ if $y_i\in\Gamma(\varpi_k)$. More formally, if   $y_i\in \Gamma(\varpi_k)$, then
$a_i\le c_k$ and therefore $y_i^{a_i}\in \Gamma(\varpi_k)^{c_k}\subset\Gamma(c_k \varpi_k)$.
Thus, $Y\in \Gamma(\lambda)$.
Note that Equation~\eqref{iota} defines the numbers $d_k=d_k(Y)$ for each vector $Y$ as above.

The next step is to consider  
$\Gamma(\varpi_k{+}\varpi_j)$ with $k\ne j$.

\begin{lm}\label{lm-gamma-2}
Suppose that $j>k$ and  $\lambda=\varpi_k+\varpi_j$. 
Then %
$$\ess(\lambda)\cap \cS(\gt r)^{\gt h_0}=\{1,  F_{2n\,1}, F_{2n\,1}^2, F_{2n\,j} F_{j\,1}\}.$$
\end{lm}
\begin{proof}
Set $\mu=\lambda|_{\gt h_0}$. Then $\dim U(\lambda,\mu)=4$.
As a representation of $\gt{sl}_2=\left<F_{2n\,1},F_{1\,1},F_{1\,2n}\right>_{\mK}$\,, it decomposes as $\mK^3\oplus\mK$.
Since $F_{2n\,1}\in\Gamma(\varpi_i)$ for each $i$, we have
$F_{2n\,1}, F_{2n\,1}^2\in\ess(\lambda)$.
It remains to show that $F_{2n\,j} F_{j\,1}$ is essential.
In the case $k=j-1$, this 
follows form the inclusion~\eqref{+} and Lemma~\ref{lm-gamma}.
Therefore suppose that $k<j-1$.
Then
$\dim V(\lambda)^{+}$ as one of the numbers
$9$, $12$, and $16$, depending on $k$ and $j$.  In any case,
$\Gamma(\lambda)$ is the disjoint union
of three subsets, $X=\{1,F_{2n\,1}, F_{2n\,k}, F_{2n\,j}, F_{k+1\,1}, F_{j+1\,1}\}$, the product
$F_{2n\,1} (X\setminus\{1\})$,
and the  subset
$$
\{ F_{2n\,k}F_{2n\,j}, \ F_{2n\,k}F_{j+1\,1}, \ F_{k+1\,1}F_{2n\,j}, \ F_{k+1\,1}F_{j+1\,1} , \ x  \},
$$
where $p x v_\lambda\in V(\lambda)_\mu^+$ and the $F_{1\,1}$-weight of $x$ is $-2$.
Since $F_{2n\,1}\in X$,
these two conditions on $x$ imply that $x=F_{2n\,t} F_{t\,1}$ for some $t\le n$.

First we show that $t\le j$. If $j<n$,  take $s>j$.  Let us regard $V(\varpi_r)$ as a subspace of 
$\bigwedge^r \CC^{2n}$, where 
$\CC^{2n}=V(\varpi_1)$ is the underlying vector space of the defining representation of $\gt g$. 
Let $\{v_1,\dots,v_{2n}\}$ be 
the standard basis of $\CC^{2n}$. 
Then
$\bv_r=v_1\wedge\ldots \wedge v_r$ is a highest weight vector  of    $V(\varpi_r)$. 
Set $u=F_{2n\,s}F_{s\,1} (\bv_k\otimes \bv_j)$. Then 
$u= \frac{1}{2}F_{2n\,1} (\bv_k\otimes \bv_j) + u'$, where
$$
u'=(v_s\wedge v_2\wedge\ldots \wedge v_k)\otimes
(v_{s'}\wedge v_2\wedge\ldots \wedge v_j) +
(v_{s'}\wedge v_2\wedge\ldots \wedge v_k)\otimes
(v_{s}\wedge v_2\wedge\ldots \wedge v_j).
$$
Here $s' > n \ge s$ and $u' = \frac{1}{2} F_{s'\,s} \tilde u$ for
$$
\tilde u= (v_s\wedge v_2\wedge\ldots \wedge v_k)\otimes
(v_{s}\wedge v_2\wedge\ldots \wedge v_j) .
$$
Thereby  $pu'=0$ by \eqref{eq-p},  hence $pu=\frac{1}{2} F_{2n\,1}(\bv_k\otimes \bv_j)$ and
$F_{s\,1}F_{2n\,s}$ is not essential for $\varpi_k+\varpi_j$.
We have shown that $x\ge F_{j\,1}F_{2n\,j}$.

Assume that $F_{j\,1}F_{2n\,j}$ is not essential. Then $w=F_{j\,1}F_{2n\,j} (\bv_k\otimes\bv_j)$
lies in the linear span   of smaller than $F_{j\,1}F_{2n\,j}$ essential monomials. Each such monomial is
of the form $m_0 m_1$, where $m_1$ has weight $-2$ w.r.t. $F_{11}$ 
and
$m_1<F_{j\,1}F_{2n\,j}$. This is possible only for $F_{2n\,1}$, $F_{2n\,j}F_{k+1\,1}$, and
$F_{j+1\,1}F_{k+1\,1}$.

The decomposition $V(\varpi_1)=\mK v_1\oplus V'(\varpi_1) \oplus \mK v_{2n}$ leads to
a $\gt g_0$-invariant tri-grading on each $V(\varpi_r)$. In the tensor product $V(\varpi_k)\otimes V(\varpi_j)$,
the vector $F_{2n\,j}F_{j\,1}  (\bv_k\otimes \bv_j)$ has non-zero summands of degrees
$$
(0,k-1,1;1,j-1,0), \ \ (0,k,0;1,j-2,1),  \ \ (0,k,0;0,j,0).
$$
The monomials $F_{2n\,j}F_{k+1\,1}$ and
$F_{j+1\,1}F_{k+1\,1}$ produce vectors of degrees
$$
(0,k,0;1,j-2,1), \ \ (0,k,0;0,j,0), \ \text{ and } \ \ (0,k,0;0,j,0).
$$
This implies that the summand of degree $(0,k-1,1;1,j-1,0)$, which is equal to
$$w=(v_{2n}\wedge v_2\wedge\ldots\wedge v_k)\otimes(v_1\wedge\ldots\wedge v_j),$$
is written as $a m_0 F_{2n\,1} (\bv_k\otimes \bv_j)$ for some $a\in\mK$ and
$m_0\in\U(\gt n_0^-)$. However,
$F_{2n\,1} (\bv_k\otimes \bv_j)=2(w+\tilde w)$, where $\tilde w\ne 0$ is of degree
$(1,k-1,0;0,j-1,1)$.
This contradiction finishes the proof.
\end{proof}

\begin{prop} \label{sp-prop}
{\sf (i)} The defining  inequalities for $\Gamma(\lambda)$  in terms of
$$
F_{2n\,1}^b (F_{2n\,n}F_{n\,1})^{b_n}\ldots(F_{2n\,2}F_{2\,1})^{b_2}  y_n^{a_n}\ldots y_2^{a_2}
$$
are:
\begin{align}
& %
0\le d_k, \  \text{ where  the numbers $d_k$ are given by \eqref{iota}}, \label{1} \\
& b_k \le d_k, \label{2} \\
&  b_k \le d_1 + \sum\limits_{i=2}^{k{-}1} (d_i-2b_i) \ \ \text{for each $k$ such that } 2\le k\le n; \label{3} \\
& b+ 2\sum\limits_{k=2}^n b_k \le \sum\limits_{i=1}^n d_i. \label{4}
\end{align}
{\sf (ii)} The semigroup $\Gamma$ is generated by
$\Gamma(\varpi_t)$ and $\Gamma(\varpi_k{+}\varpi_j)$ with $1\le t,k,j\le n$ and $j> k+1$.
\end{prop}
\begin{proof}
{\sf (i)} The inequalities \eqref{1} are equivalent to $U(\lambda,\mu)\ne 0$, where
$\mu$ is the $\gt h_0$-weight of $y_n^{a_n}\ldots y_2^{a_2} v_\lambda$.
Each weight $\mu$ such that $U(\lambda,\mu)\ne 0$ defines the tuple $\bar a=(a_2,\ldots,a_n)$ uniquely.  
Let $\bar a$ be fixed.

Next we show that the number of tuples $(b,b_n,\ldots,b_2)\in\mathbb Z^n_{+}$ satisfying the
inequalities \eqref{2}--\eqref{4} is equal to $\prod\limits_{i=1}^{n} (d_i+1)=\dim U(\lambda,\mu)$. We argue by induction on
$n$. If $n=1$, then there is just one inequality $b\le d_1$.  There are $d_1+1$ possibilities for $b$.

Suppose that $n=2$. Then $b_2\le d_1,d_2$.
Each admissible $b_2$ corresponds to the
irreducible $\gt{sl}_2$-submodule of $\cS^{d_1}\mathbb C^2\otimes \cS^{d_2}\mathbb C^2$
of dimension $d_1+d_2+1 -2b_2$. If $b_2$ is fixed, then there are exactly
$d_1+d_2-2b_2+1$ possibilities for $b$. For $n=2$, the number of tuples $(b,b_2)$ is correct.

Suppose now that $n>2$ and that for $n-1$ there is a bijection between
the tuples $\bar b=(b_2,\ldots,b_{n-1})$ satisfying the inequalities and the irreducible $\gt{sl}_2$-submodules of
$$
\cS^{d_1}\mathbb C^2 \otimes \ldots\otimes \cS^{d_{n-1}}\mathbb C^2
$$
such that the module $V(\bar b)$ corresponding to $\bar b$ is of dimension
$$\sum\limits_{i=1}^{n-1} d_i + 1 - 2\sum\limits_{i=2}^{n-1} b_i.$$
The irreducible submodules  of $V(\bar b)\otimes \cS^{d_n}\mK^2$ can be enumerated  by
integers
$b_n$ such that $$0\le b_n\le \min(d_n, \dim V(\bar b)-1).$$
We can arrange the submodules in such a way that the dimension decreases when $b_n$ increases.
Then $b_n$, or rather $(b_2,\ldots,b_{n-1},b_n)$, corresponds to the summand of dimension
$$\sum\limits_{i=1}^{n} d_i + 1 - 2\sum\limits_{i=2}^{n} b_i.$$
 This completes the  inductive argument. 

In the proof of  part {\sf (ii)} below, we show that each admissible tuple
$$(a_2,\ldots,a_n,b,b_2,\ldots,b_n)$$ defines a monomial of  $\Gamma(\lambda)$.
Hence by the dimension reasons,
{\sf (i)}  holds.

{\sf (ii)} For convenience, we will identify the monomials $m_1\in\Gamma(\lambda)$ with the
tuples of their exponents and use additive notation for $\Gamma(\lambda)$,
so that
$\Gamma(\lambda)+\Gamma(\lambda')\subset \Gamma(\lambda+\lambda')$; see \eqref{+}.

Let $(\bar b,\bar a)$ with $\bar b=(b,b_2,\ldots,b_{n})$, $\bar a=(a_2,\ldots,a_n)$ be an admissible tuple.
Recall that each $y_i$ belongs to a unique $\Gamma(\varpi_s)$ with $s=s(i)$.
Set $\tilde\lambda=\lambda-\sum\limits_{i=2}^n a_i \varpi_{s(i)}$.
In view of \eqref{iota}, we have $\tilde\lambda=\sum\limits_{i=1}^{n-1} d_i \varpi_i$.
The inequalities \eqref{1} guarantee that $\tilde\lambda$ is a dominant weight of  $\gt g$.
If $\bar b$, identified with $(\bar b,\bar 0)$, lies in $\Gamma(\tilde\lambda)$, then
$$
(\bar b,\bar a)\in \Gamma(\tilde\lambda)+\sum_{i=2}^n a_i \Gamma(\varpi_{s(i)})\subset\Gamma(\lambda).
$$
Next we express $\bar b$  as a sum of tuples belonging to  sets
$\Gamma(\varpi_t)$ and $\Gamma(\varpi_k+\varpi_j)$ and show that indeed
$\bar b\in\Gamma(\tilde\lambda)$.

If all $d_k$ are zero, then $\bar b=0$ and  there is nothing to prove.
Suppose next that $d_k\ne 0$ only for $k=i$. Then $b_j=0$ for all $j\ge 2$ and
only $F_{2n\,1}^b$ with $b\le d_i$ is left. Here $\bar b\in d_i \Gamma(\varpi_i)$.
The proof continues by induction on $|\bar b|=b+b_2+\ldots + b_n$.

Let $k<r$ be the smallest integers such that $d_r,d_k\ne 0$.
Note that $b_2=\ldots=b_{r-1}=0$. If all $b_i$ with $i\ge 2$ are equal to zero, then
$\bar b=(b,0,\ldots,0)$, where $b\le \sum\limits_{i=1}^n d_i$. Again, such $\bar b$ belongs
to $\sum\limits_{i=1}^{n} d_i \Gamma(\varpi_i)\subset\Gamma(\tilde\lambda)$. Therefore assume that $\bar b\ne (b,0,\ldots,0)$.

Let $j\ge r$ be the smallest integer  such that $b_r\ne 0$.
We divide  our monomial by $F_{2n\,j}F_{j\,1}$, which is an element of   $\Gamma(\varpi_k{+}\varpi_j)$ by Lemma~\ref{lm-gamma-2}.  Note that in case $j=k+1$, we have
$F_{2n\,j}F_{j\,1}\in \Gamma(\varpi_j)\Gamma(\varpi_k)$.
The division corresponds to replacing
$\bar b$ with  $\bar b'=(b,b_2',\ldots,b_n')$, where $b_i'=b_i$ for $i\ne j$ and  $b_j'=b_j-1$.
Accordingly, set $\lambda'=\tilde\lambda-(\varpi_k+\varpi_j)$.
We have $\lambda'=\sum\limits_{i=1}^n {d_i'} \varpi_i$, where $d_i'=d_i$ for $i\ne k,j$ and
$d_i'=d_i-1$ for $i\in\{k,j\}$.
The next task is to see that the inequalities \eqref{2}--\eqref{4}  hold for $\bar b'$ and $\lambda'$.

Consider \eqref{2}. For $i\ne k,j$, we have $b_i'=b_i\le d_i=d_i'$.
If $k=1$, then there is no  $b_k$. If  $k\ge 2$, then $b_k'=b_k=0$ and $b_k\le d_k'$.
Finally,  $b_j'=b_j-1\le d_j-1=d_j'$. These inequalities hold.

Consider \eqref{3}.  For $s<j$, we have $b_s=0$.
 Clearly, the inequalities hold for all such $s$. 
For the index  $j$, we have
$$
b_j'\le \left(\sum\limits_{t=1}^{j-1} d_t\right) -1 = \sum\limits_{t=1}^{j-1} d_t' = d_1' + \sum\limits_{t=2}^{j-1} (d_t'-2b_t').
$$
 For $s>j$, the new right hand side
$d_1'+\sum\limits_{t=2}^{s-1} (d_t'-2b_t')$ is equal to the old one. Since $b_s'=b_s$ here, all
the inequalities hold.

Finally, consider \eqref{4}. We have
\ben
\sum\limits_{i=1}^n {d_i'}-2\sum\limits_{t=2}^{n} b_t'= \sum\limits_{i=1}^n {d_i}-2\sum\limits_{t=2}^{n} b_t.
\een
Hence the inequality for $b$ holds.

Summing up, $\bar b'$ belongs to $\Gamma(\lambda')$, because $|\bar b'|< |\bar b|$, and hence  $\bar b$ belongs to $\Gamma(\tilde\lambda)$.
\end{proof}

The perspective on $U(\lambda,\mu)$ taken  in this section differs from the usual one.
In order to obtain a basis, we have regarded $U(\lambda,\nu)$ as a direct sum
of $\gt{sl}_2$-modules instead of a tensor product. On the one side,
this leads to a more complicated set of inequalities, on the other, we are getting one more basis. 

Set $\tilde p=p_{\gt{sl}_2} p$, where $p_{\gt{sl}_2}$ is the extremal projector associated with
$\gt{sl}_2=\left<F_{2n\,1},F_{1\,1},F_{1\,2n}\right>_{\mK}$ and $p$ is the projector
of $\gt g_0$ as before.  Let us restrict $V(\lambda)$ to $\gt g_0\oplus\gt{sl}_2$.

\begin{cl}
The subspace  $V(\lambda)^+\cap V(\lambda)^{F_{1\,2n}}$ has a basis
$$
\{\tilde p m_1 v_\lambda \mid m_1\in \Gamma(\lambda) \text{ is given by exponents} \ (0,b_2,\ldots,b_n,a_2,\ldots,a_n)\},
$$
i.e., we are taking the subset, where $b=0$.
\end{cl}

The chain of subalgebras \eqref{sp-chain}
can be used in order to extend the basis of Proposition~\ref{sp-prop}  to a basis
for $V(\lambda)$.

\section{Relations to the Gelfand--Tsetlin basis } 
\label{sec:rgt}

We start by recalling a construction of the celebrated basis of Gelfand and Tsetlin~\cite{gt:fdu}
for each finite-dimensional irreducible representation $L(\la)$ of $\gl_n$
as defined in the Introduction. We refer the reader to the review paper~\cite{m:gtb}
where several such constructions are discussed. To be consistent with the notation of that paper
we will now let $\xi$ denote the highest weight vector of $L(\la)$ (along with $v_{\lambda}$).

Consider the extremal projector $p$ associated with the Lie algebra $\gl_{n-1}$.
Recall that the {\em Mickelsson--Zhelobenko algebra} $\Zr(\gl_n,\gl_{n-1})$
is generated by the elements $E_{n\tss n}$,
$pE_{i\tss n}$ and $pE_{n\tss i}$ with $i=1,\dots,n-1$; see \cite[Sect.~2.3]{m:gtb}
for the definitions. The {\em lowering operators}
are elements of the universal enveloping algebra $\U(\gl_n)$ which can be defined by
the formulas
\beql{alow2}
z_{n\tss k}=pE_{n\tss k}\ts
(h_k-h_{k+1})\cdots (h_k-h_{n-1}),
\eeq
where $h_k=E_{k\tss k}-k+1$. By the branching rule,
the restriction of $L(\la)$ to the subalgebra $\gl_{n-1}$ is isomorphic
to the direct sum of irreducible $\gl_{n-1}$-modules $L'(\mu)$,
\ben
L(\lambda)|^{}_{\gl_{n-1}}\simeq \underset{\mu}\oplus\ts L'(\mu),
\een
summed over the highest
weights $\mu=(\mu_1,\dots,\mu_{n-1})$ satisfying the betweenness conditions
\beql{amulam2}
\lambda_i-\mu_i\in\ZZ_+ 	\quad\text{and}\quad \mu_i-\lambda_{i+1}\in
\ZZ_+ 	\qquad\text{for}\quad i=1,\dots,n-1.
\eeq
The $\gl_{n-1}$-submodule in $L(\lambda)$ isomorphic to $L'(\mu)$
is generated by the vector
\ben
\xi_{\mu}=z_{n\tss 1}^{\la_1-\mu_1}\dots z_{n\ts n-1}^{\la_{n-1}-\mu_{n-1}}\ts\xi.
\een
In the next lemma we suppose that each of the highest weights $\mu$ and $\mu'$ satisfies conditions \eqref{amulam2}  and we use
the lexicographical ordering $\succ$ on such weights,
where for complex numbers $a$ and $b$ we
assume that $a\geqslant b$ if and only if $a-b\in \ZZ_+$.

\ble\label{lem:hwcomp}
For any given $\mu$, in the module $L(\la)$ we have
\ben
E_{n\tss 1}^{\la^{}_{1}-\mu^{}_{1}}
\dots E_{n\ts n-1}^{\la^{}_{n-1}-\mu^{}_{n-1}}\tss\xi =
c\ts\tss \xi_{\mu}
+\sum_{\mu'\succ\mu}\ts u(\mu')\ts
\xi_{\mu'}
\een
for a nonzero constant $c$ and some elements $u(\mu')\in\U(\n^-_0)$,
where the sum is taken over the highest weights $\mu'$
satisfying conditions \eqref{amulam2}.
\ele

\bpf
Starting from the rightmost
generator which occurs in the product on the left hand side and
proceeding to the left, we use the
{\em inversion formula}
\ben
E_{n\tss k}=pE_{n\tss k}+\sum_{k<k_1<\cdots<k_s<n}
E_{k_1k}E_{k_2k_1}\dots E_{k_sk_{s-1}}\ts p\tss E_{n\tss k_s}
\frac{1}{(h_{k_s}-h_{k})(h_{k_s}-h_{k_1})\cdots (h_{k_s}-h_{k_{s-1}})},
\een
summed over $s=1,2,\dots$. Arguing by induction, observe that
each generator $E_{n\tss l}$ with $l\leqslant k$ commutes with
all factors $E_{k_1k},E_{k_2k_1},\dots, E_{k_sk_{s-1}}$ so that
the proof is completed by using
\eqref{alow2} and taking into account the fact that the lowering operators $z_{n\tss k}$
pairwise commute.
\epf

The vectors $\xi^{}_{\Lambda}$ of the Gelfand--Tsetlin basis  $\{\xi_\Lambda\}$
of $L(\la)$ are parameterised by the patterns $\La$ defined in the Introduction.
They are found by the formula
\beql{axiL}
\xi^{}_{\Lambda}=\prod_{k=2,\dots,n}^{\longrightarrow}
\Bigl(z_{k\tss 1}^{\lambda^{}_{k\tss 1}-\lambda^{}_{k-1\ts 1}}\dots
z_{k\ts k-1}^{\lambda^{}_{k\ts k-1}-\lambda^{}_{k-1\ts k-1}}\Bigr)\ts\xi.
\eeq
Represent each pattern $\La$ associated with $\la$ as the sequence of its rows:
\ben
\La=(\bar\la_{n-1},\dots,\bar\la_{1}),\qquad \bar\la_{k}=(\la_{k1},\dots,\la_{kk}),
\een
and consider the lexicographical ordering $\succ$ on the sequences by using
the ordering on the highest weights introduced above.
Recall the vectors $\pi^{}_{\La}$ defined in Theorem A. We now obtain a  proof
of this theorem.

\bpr\label{prop:triang}
For each pattern $\La$ associated with $\la$,
in the module $L(\la)$ we have
\ben
\pi^{}_{\La}=\sum_{\La'\succcurlyeq \La} c_{\La,\La'}\tss \xi^{}_{\La'}
\ \ \text{ and hence } \ \
\xi_{\Lambda} = \sum_{\La'\succcurlyeq \La} d_{\La,\La'}\tss \pi^{}_{\La'}
\een
for some constants $c_{\La,\La'}$ and $d_{\Lambda,\Lambda'}$, whereby $c_{\La,\La}=d_{\La,\La}^{-1}\ne 0$.
\epr

\bpf
Due to the inductive structure of the vectors \eqref{axiL}, the proposition follows
by a repeated application of Lemma~\ref{lem:hwcomp}.
\epf

\subsection{The PBW-parameterisation of the canonical basis}  \label{L-PBW}
The canonical basis for  $V(\lambda)$ constructed by Lusztig \cite{L1,L2}
has a
PBW-parameterisation, which fits into the FFLV-framework.

Let $\omega_0=s_{i_1}\ldots s_{i_N}$ be a reduced decomposition of the longest element
$\omega_0\in W(\gt g,\gt h)$ of the Weyl group.
Define the sequence of positive roots $\beta_1,\ldots,\beta_N$ by
$\beta_k=s_{i_1}\ldots s_{i_{k-1}} (\alpha_{i_k})$, where $\alpha_r$ is the $r$th simple root.
Then $\beta_t\ne \beta_k$ for $k\ne t$, see e.g. \cite[Sect.~12]{xfp}.
Let $f_k$ be the negative root vector corresponding to ${\beta_k}$. Make use of  the
{\it  right opposite} lexicographical order on the monomials
$f_1^{a_1}\ldots f_N^{a_N}$, which means that
$f_1^{a_1}\ldots f_N^{a_N}< f_1^{a_1'}\ldots f_N^{a_N'}$ if and only if there is $k$ such that
$1\le k\le N$ and
$$
a_k>a_k', \ \ a_r=a_r' \ \text{ for } \ r>k.
$$
Use the same sequence of vectors for the elements of $\U(\gt n^-)$.
Then the elements of the canonical basis for $V(\lambda)$ are in bijection
with $\ess(\lambda)$. Moreover, if the element $B(m)v_\lambda$ of the canonical basis corresponds
to $m\in\ess(\lambda)$, then
\begin{equation} \label{can}
B(m)v_\lambda \in m v_\lambda  + \left< \tilde m v_\lambda \mid \tilde m < m\right>_{\mK},
\end{equation}
see e.g. \cite[Sect.~12]{xfp}.  Note that we have omitted the ``height weighted function $\Psi$"  of \cite{xfp}
on the monomials, because it becomes redundant once one fixes a finite-dimensional module
$V(\lambda)$.

Let $B_\lambda$ be the canonical basis of $V(\lambda)$. 
Then  the dual basis $B^*_\lambda\subset V(\lambda^*)$ is
 {\it good} in the terminology of 
 \cite{bz-duke} by \cite[Theorem~4.4]{L2}.

\begin{ex} Let $\gt g$ be of type {\sf A}$_{n-1}$. Choose
the decomposition
$\omega_0= s_1 s_2s_1 \ldots s_{n-1}\ldots s_2s_1$. Then
$$
f_1 \ldots f_N = E_{2\,1} E_{3\,1} E_{3\,2} \ldots E_{n\,1} E_{n\,2} \ldots E_{n\,n-1}.
$$
The right opposite lexicographical order satisfies the assumptions of Section~\ref{sec-rules} at each
step of the reductions along the Gelfand--Tsetlin chain of subalgebras. Therefore, we get the
basis $\{\pi_{\Lambda}\}$ described in Theorem~A.
Note that this basis was  obtained in \cite{Markus-dis} for the same $\omega_0$ as above.
\end{ex}

Keep the assumption $\gt g=\gt{gl}_n$. By the weight considerations, we have
\begin{equation}\label{p-mu}
\xi_\mu =d_{\mu} \,p E_{n\,1}^{\lambda_1-\mu_1} \ldots E_{n\,n-1}^{\lambda_{n-1}-\mu_{n-1}} v_{\lambda}
\end{equation}
for some $d_{\mu}\in\mK^{^\times}$.

\begin{cl}\label{cor:tramatwo}
 For each dominant $\lambda$, the transition matrix between the canonical
and the Gelfand--Tsetlin bases of $L(\lambda)$ is triangular.
\end{cl}
\begin{proof}
Let $\Lambda$ be a Gelfand--Tsetlin pattern associated with
$\lambda$. Consider $\pi_\Lambda=m v_\lambda$, where   $m=m(\Lambda)$ 
is the same as in Theorem~A.

If we use the right opposite lexicographical order as above, then \eqref{eq-p2} holds
for all reduction steps along the Gelfand--Tsetlin chain of subalgebras.
For each step, the analogue of \eqref{p-mu} holds as well.
 Therefore $\pi_\Lambda$  is the leading term of $c_{\Lambda,\Lambda} \xi_{\Lambda}$.
In view of  \eqref{can}, $\pi_\Lambda$ is also the leading term of
  $B(m)\xi$.
\end{proof}

\begin{rmk}
In the case $n=3$ the canonical basis is monomial \cite[Example~3.4]{L1}. 
Thereby this particular case of
Corollary~\ref{cor:tramatwo} follows by a simple calculation with the use of the Gelfand--Tsetlin
formulas.
\end{rmk}

\begin{prop} \label{tri-dual}
There is an enumeration  of  the elements  $\xi_\Lambda$ such that
the transition matrix between $B^*_{\lambda^*}\subset L(\lambda)$ and  $\{\xi_\Lambda\}$
is triangular.
\end{prop}
\begin{proof}
The dual basis $\{\xi_\Lambda^*\}\subset L(\lambda^*)$ is the Gelfand--Tsetlin basis of
$L(\lambda^*)$ up to a permutation of its elements and multiplications by non-zero scalars.
By Corollary~\ref{cor:tramatwo}, the transition matrix between $B_{\lambda^*}$ and
$\{\xi_{\Lambda}^*\}$ is triangular. Therefore
$B^*_{\lambda^*}$  and  $\{\xi_\Lambda\}$ are related by a triangular matrix as well.
\end{proof}

Outside type {\sf A}, these PBW-type bases become less transparent, see e.g. \cite{Markus-dis}.

\begin{ex}
Let $\gt g$ be of type {\sf C}$_n$. Choose the decomposition
$$
\omega_0=s_n s_{n-1}s_n s_{n-1} \ldots s_1s_2\ldots s_{n-1} s_n s_{n-1} \ldots s_2 s_1\, .
$$
Then $f_k\in\gt{sp}_{2n-2}$ for $k\le N-2n+1$ and
$$
f_{N-2n+2} \ldots f_N = F_{2n\,2}\ldots F_{2n\,n} F_{2n\,1} F_{2n\,n+1} \ldots F_{2n\,2n-1}\, .
$$
It is not difficult to see that such a choice produces a branching semigroup related to
$\gt{sp}_{2n} \downarrow \gt{sp}_{2n-2}$ and that this semigroup is the same as in
Section~\ref{sp2}.
\end{ex}

\subsection{Monomials in simple root vectors}
 \label{subsec:sb}
 The bases of Littelmann \cite{l:cc} arise as different  parameterisations of the canonical basis.
His construction
 involves branching and produces a basis
of $V(\lambda)$ by applying iterated negative simple root vectors to $v_\lambda$, see \cite{l:cc}
and also \cite[Sect.~11]{xfp} for a connection with the FFLV-method.
In type {\sf A}, the construction is most transparent 
\cite{peter-A}, \cite[Sect.~5\&10]{l:cc}. 

Set $f_k=E_{k+1\,k}$. The subspace $L(\lambda)^{+}$ is the linear span
of vectors $p f_{n-1}^{a_{n-1}}\ldots f_1^{a_1} v_{\lambda}$, where $a_{n-1}\ge a_{n-2}\ge \ldots\ge a_1$.
Set $a_0=0$.
By the weight considerations,
$$
\xi_{\mu}=p f_{n-1}^{a_{n-1}}\ldots f_1^{a_1} v_{\lambda} \ \text{ with } \
a_k-a_{k-1} = \lambda_k-\mu_k,
$$
up to a non-zero scalar.
In view of
the equality
$$
[f_{n-1},[f_{n-2},\dots, [f_{k+1} ,f_{k}]\dots ]] = E_{n\,k},
$$
we can conclude directly, without weight arguments,  that
\begin{equation}\label{p-string}
p f_{n-1}^{a_{n-1}}\ldots f_1^{a_1} v_\lambda =
p E_{n\,n-1}^{a_{n-1}-a_{n-2}} \ldots E_{n\,1}^{a_1} v_{\lambda}\,.
\end{equation}

A basis of $L(\la)$ is obtained
inductively, omitting extremal projectors, so that the basis vectors have the form
$$
\boldsymbol{f}(\Lambda)=f_1^{a_{n-1,1}} f_{2}^{a_{n-2,2}} f_1^{a_{n-2,1}}
\ldots  f_{n-1}^{a_{1,n-1}}\ldots f_1^{a_{1,1}} v_{\lambda}
$$
and are naturally parameterised by the Gelfand--Tsetlin patterns, see \cite[Corollary~5]{l:cc}.
In the notation of \eqref{aconl},
\begin{equation} \label{l-a}
a_{k,j}=\sum_{i=1}^{j} (\lambda_{n-k+1\,i}-\lambda_{n-k\,i}).
\end{equation}
Let $m(\Lambda)$ be the leading term of $\boldsymbol{f}(\Lambda)$ in
the monomial order used in Section~\ref{sec-A-proofs}.  Combining \eqref{eq-p2} with \eqref{p-string}, we see that  $m(\Lambda)\xi=\pi_\Lambda$ and that again $\pi_\Lambda$ is the leading term of
$c_{\Lambda,\Lambda}\xi_\Lambda$.  Summing up,
$$
\boldsymbol{f}(\Lambda)   \in \pi_\Lambda + \left< m\xi \mid m < m(\Lambda)\right>_{\mK} =
c_{\Lambda,\Lambda} \xi_\Lambda + \left< m\xi \mid m < m(\Lambda)\right>_{\mK}
$$
with a non-zero  $c_{\Lambda,\Lambda}\in\mK$. Therefore,
the transition matrices  between all three bases are
triangular. 

\begin{rmk}
Relations between different monomial bases parameterising  the canonical basis are
studied in \cite{Chari-xi}. The fact that the bases
$\{  \boldsymbol{f}(\Lambda) \}$ and $\{\pi_{\Lambda}\}$ are related by a unitary matrix can be deduced
 from the results of that  paper.
\end{rmk}

\begin{rmk}\label{rem:trama}
Let $\succ$ be the lexicographical order on $\mathbb Z^{N}$.
Choose the enumeration of the basis vectors $\boldsymbol{f}(\Lambda)$ is such a way
that the corresponding  sequences
$$
\bar a=\bar a(\Lambda)= (a_{n-1,1},a_{n-2,2},a_{n-2,1}\ldots,a_{1,1}),
$$
see \eqref{l-a}, are
decreasing w.r.t  the order $\succ$.  Then
the transition matrix between $\{\boldsymbol{f}(\Lambda)\}$ and the canonical basis of
 $L(\lambda)$ is upper triangular  and unipotent  by \cite[Prop.~10.3]{l:cc}. 
Refining the above considerations, one can show that
$$
c_{\Lambda,\Lambda} \xi_{\Lambda} \in \boldsymbol{f}(\Lambda) + \left< \boldsymbol{f}(\Lambda') \mid
\bar a(\Lambda') \succ \bar a\right>_{\mK}
$$
and thus produce a different proof of Corollary~\ref{cor:tramatwo}.
  \end{rmk}

\section{A Gelfand--Tsetlin-type basis for representations of $\spa_{2n}$}
\label{sec:wb}

We now aim to prove
an analogue of Proposition~\ref{prop:triang} for the symplectic Lie algebra
$\spa_{2n}$. The vectors $\theta^{}_{\La}$ defined in Theorem B turn out to be
related to a certain modification of the basis of \cite{m:br}.
In this section we will rely on the exposition in \cite[Ch.~9]{m:yc} to produce this
modification.

Given a  type {\sf C}
pattern $\La$
associated with
$\la$, as defined in the Introduction, set
\beql{lunivpatt}
l_{k\tss i}=\la_{k\tss i}-i-\frac12,\qquad l'_{k\tss i}=\la'_{k\tss i}-i+\frac12.
\eeq

\bth\label{thm:genactc}
The $\spa_{2n}$-module $V(\la)$ admits a basis $\ze^{}_{\La}$
parameterised by the
 type {\sf C}
patterns $\La$
associated with
$\la$ such that
the action of generators of $\spa_{2n}$ in the basis
is given by the formulas
\ben
\bal
F_{k\tss k}\ts \ze^{}_{\La}&=\Bigg(\sum_{i=1}^k\la^{}_{k\tss i}+\sum_{i=1}^{k-1}\la^{}_{k-1\ts i}
-2\sum_{i=1}^k\la'_{k\tss i}\Bigg)
\ze^{}_{\La},\\
F_{k,-k}\ts \ze^{}_{\La}
&=\sum_{i=1}^k A_{k\tss i}\ts
\ze^{}_{\La-\de'_{k\tss i}},\qquad\quad
F_{-k,k}\ts \ze^{}_{\La}
=\sum_{i=1}^k B_{k\tss i}\ts
\ze^{}_{\La+\de'_{k\tss i}},\\
F_{k-1,-k}\ts \ze^{}_{\La}&=-\sum_{i=1}^{k-1}C_{k\tss i}\ts
\ze^{}_{\La+\de^{}_{k-1\ts i}}
-\sum_{i=1}^k\sum_{j,m=1}^{k-1}D_{k\tss i\tss j\tss m}\ts
\ze^{}_{\La-\de'_{k\tss i}-\de^{}_{k-1\ts j}-\de'_{k-1\ts m}},
\eal
\een
where
\ben
\bal
A_{k\tss i}&=\prod_{a=1, \ts a\ne i}^k\frac{1}{l'_{k\tss a}-l'_{k\tss i}},\\
B_{k\tss i}&=2\ts A_{k\tss i}\ts \big(2\ts l'_{k\tss i}+1\big)\ts
\prod_{a=1}^k \big(\tss l^{}_{k\tss a}-l'_{k\tss i}\big)\ts\prod_{a=1}^{k-1}
\big(\tss l^{}_{k-1\ts a}-l'_{k\tss i}\big),\\
C_{k\tss i}&=\frac{1}{2\ts l^{}_{k-1\ts i}+1}\prod_{a=1, \ts a\ne i}^{k-1}
\frac{1}{\big(\tss l_{k-1\ts i}-l_{k-1\ts a}\big)
\big(\tss l_{k-1\ts i}+l_{k-1\ts a}+1\big)},
\eal
\een
and
\ben
D_{k\tss i\tss j\tss m}=A_{k\tss i}\tss A_{k-1\ts m}\tss C_{k\tss j}
\prod_{a=1, \ts a\ne i}^{k}\big(\tss l_{k-1\ts j}^2-l^{\prime\tss 2}_{k\tss a}\big)
\ts\prod_{a=1, \ts a\ne m}^{k-1}
\big(\tss l_{k-1\ts j}^2-l^{\prime\tss 2}_{k-1\ts a}\big).
\een
The arrays $\La\pm\de^{}_{k\tss i}$ and $\La\pm\de'_{k\tss i}$
are obtained from $\La$ by replacing $\la^{}_{k\tss i}$ and $\la'_{k\tss i}$
by $\la^{}_{k\tss i}\pm1$ and $\la'_{k\tss i}\pm1$ respectively. The vector
$\ze^{}_{\La}$ is considered to be zero if the array $\La$ is not a pattern.
\eth

\bpf
The proof is not essentially different from that of \cite[Theorem~9.6.2]{m:yc},
so we only point out some key steps and alternative choices made in the arguments.

Suppose that $\mu=(\mu_1,\dots,\mu_{n-1})$ is an $\spa_{2n-2}$-highest weight.
The multiplicity space
$V(\la)^+_{\mu}$ is nonzero
if and only if the components
of $\la$ and $\mu$ satisfy the inequalities
\beql{ineqtwmhwc}
\la_i\geqslant \mu_{i+1},\quad i=1,\dots,n-2\Fand
\mu_i\geqslant \la_{i+1},\quad i=1,\dots,n-1.
\eeq
When it is nonzero,
the vector space $V(\la)^+_{\mu}$ carries an irreducible representation
of the {\em twisted Yangian} $\Y(\spa_2)$.
By \cite[Theorem~9.4.11]{m:yc}, this representation is isomorphic to
the tensor product,
\beql{isomtp}
V(\la)^+_{\mu}\cong
L(\al_1,\be_1)\ot\dots\ot L(\al_n,\be_n),
\eeq
where
\ben
\al_i=\min\{\la_{i-1},\mu_{i-1}\}-i+\frac12,\qquad
\be_i=\max\{\la_{i},\mu_{i}\}-i+\frac12,
\een
assuming that $\la_0=\mu_0=0$ and $\max\{\la_{n},\mu_{n}\}$
is understood as being equal to $\la_n$.
Each factor $L(\al_i,\be_i)$ is the highest weight $\gl_2$-module which is extended
to the {\em evaluation module} over the {\em Yangian} $\Y(\gl_2)$.
The coproduct on the Yangian allows one to equip
the tensor product in \eqref{isomtp}
with a $\Y(\gl_2)$-module structure. This module is then restricted
to the subalgebra
$\Y(\spa_2)\subset \Y(\gl_2)$.

The required modification of the construction relies on \cite[Corollary~4.3.5]{m:yc}
which implies an alternative isomorphism
\beql{isomtpalt}
V(\la)^+_{\mu}\cong
L(-\be_1,-\al_1)\ot\dots\ot L(-\be_n,-\al_n).
\eeq
Although the tensor products in \eqref{isomtp} and \eqref{isomtpalt}
are isomorphic as $\Y(\spa_2)$-modules, they differ as $\Y(\gl_2)$-modules.
As we shall see below, the use of the alternative isomorphism leads to
a different basis of the multiplicity space $V(\la)^+_{\mu}$.

The basis vectors of $V(\la)^+_{\mu}$ will be constructed with the use of the
{\em Mickelsson--Zhelobenko algebra} $\Zr(\spa_{2n},\spa_{2n-2})$.
The {\em lowering operators} are elements of $\Zr(\spa_{2n},\spa_{2n-2})$
defined by
\beql{normali}
z_{i,-n}=pF_{i,-n}(f_i-f_{i-1})\dots(f_i-f_{-n+1}),\qquad i=-n+1,\dots,n-1,
\eeq
where $p$ is the extremal projector for $\spa_{2n-2}$, and we set
\ben
f_i=F_{ii}-i,\qquad f_{-i}=-F_{ii}+i,
\een
for all $i=1,\dots,n$. One more lowering operator $z_{n,-n}\in \Zr(\spa_{2n},\spa_{2n-2})$
is defined by
\ben
z_{n,-n}=
\sum_{n>i_1>\cdots>i_s>-n}
F_{n\tss i_1}F_{i_1i_2}\cdots F_{i_s,-n}\ts
(f_{n}-f_{j_1})\cdots (f_{n}-f_{j_k}),
\een
where $s=0,1,\dots$ and $\{j_1,\dots,j_k\}$ is the complement
to the subset $\{i_1,\dots,i_s\}$ in the set
$\{-n+1,\dots,n-1\}$. We will also need an
interpolation polynomial $Z_{n,-n}(u)$ with coefficients
in the Mickelsson--Zhelobenko algebra given by
\beql{cZn-n}
Z_{n,-n}(u)=F_{n,-n}\tss
\prod_{i=-n+1}^{n-1}(u+g_i)-\sum_{i=-n+1}^{n-1}z_{n\tss i}\tss z_{i,-n}
\prod_{j=-n+1,\ts j\ne  i}^{n-1}\frac{u+g_j}{g_i-g_j},
\eeq
where $g_i=f_i+1/2$ for all $i$ and we set $z_{n\tss i}=(-1)^{n-i}\ts z_{-i,-n}$.
This polynomial is even in $u$ and has the properties
\beql{evalzabgi}
Z_{n,-n}(-g_i)=z_{n\tss i}\tss z_{i,-n},\qquad i=-n+1,\dots,n-1,
\eeq
and
\beql{zama}
Z_{n,-n}(-g_n)=z_{n,-n}.
\eeq

Recall that the dimension of the multiplicity space $V(\la)^+_{\mu}$
equals the number
of $n$-tuples of integers $\nu=(\nu_1,\dots,\nu_n)$
satisfying the {\em betweenness conditions} \eqref{eq-nu}.
Let us set
\beql{deflga}
\ga_i=\nu_i-i+\frac12,\qquad i=1,\dots,n.
\eeq
The highest vector of the $\Y(\spa_2)$-module $V(\la)^+_{\mu}$
is given by the formula (it coincides with the vector in \cite[(9.69)]{m:yc} up to a sign):
\beql{bcdhv}
\xi_{\mu}=\prod_{i=1}^{n-1}\Big(z_{n,-i}^{\max\{\la_i,\mu_i\}-\la_i}
z_{-i,-n}^{\max\{\la_i,\mu_i\}-\mu_i}\Big)\ts \xi,
\eeq
so that following the proof of \cite[Theorem~9.5.1]{m:yc}
and using the isomorphism \eqref{isomtpalt} instead of \eqref{isomtp},
we find that the vectors
\beql{cvlm}
\prod_{i=1}^{n}Z_{n,-n}(\ga_i+1)\dots
Z_{n,-n}(\al_i-1)\tss Z_{n,-n}(\al_i)\ts\xi_{\mu}
\eeq
with $\nu$ satisfying the betweenness conditions
form a basis of $V(\la)^+_{\mu}$. By repeating the argument of that proof,
we can conclude that the vectors
\beql{cxinu}
\xi_{\nu}=
\prod_{i=1}^{n-1}z_{n,-i}^{\mu_{i}-\nu_{i+1}}
z_{-i,-n}^{\la_{i}-\nu_{i+1}}\cdot
Z_{n,-n}(\ga_1+1)\dots \tss Z_{n,-n}(\al_1)\ts\xi
\eeq
parameterised by the
$n$-tuples $\nu$
satisfying the betweenness conditions
form a basis of the multiplicity space $V(\la)^+_{\mu}$.

Taking into account the decomposition \eqref{dectp} and applying the same argument
to the subalgebras of the chain \eqref{sp-chain},
we obtain that the vectors
\ben
\xi^{}_{\La}=
\prod_{k=1,\dots,n}^{\longrightarrow}
\Big(\prod_{i=1}^{k-1}
z_{k,-i}^{\la^{}_{k-1\ts i}-\la'_{k\ts i+1}}
z_{-i,-k}^{\la^{}_{k\tss i}-\la'_{k\ts i+1}}
\cdot Z_{k,-k}(\la'_{k\ts 1}+1/2)\dots \tss Z_{k,-k}(-1/2)\Big)\ts\xi
\een
parameterised by all
patterns $\La$ associated with $\la$ form a basis of
the representation $V(\la)$ of $\spa_{2n}$. The same calculations as in the proof
of \cite[Theorem~9.6.2]{m:yc} allow one to get the formulas for the action
of the generators of the Lie algebra $\spa_{2n}$ in the basis $\xi^{}_{\La}$ and then write
them in terms of the normalised basis
vectors
\ben
\ze^{}_{\La}=\prod_{k=2}^n\ts\prod_{1\leqslant i<j\leqslant k}
\frac{1}{(-l'_{k\tss i}-l'_{k\tss j}-1)\tss !}\ts \xi^{}_{\La}
\een
thus completing the proof.
\epf

\bre\label{rem:symm}
When written for the basis vectors $\xi^{}_{\La}$, the matrix elements
of the generators of $\spa_{2n}$ provided by
Theorem~\ref{thm:genactc} and those of
\cite[Theorem~9.6.2]{m:yc}
exhibit the following symmetry: the formal replacements $\La\mapsto -\La$
together with
$l^{}_{k\tss i}\mapsto -l^{}_{k\tss i}$ and $l'_{k\tss i}\mapsto -l'_{k\tss i}$
transform the matrix elements from one case to the other.
\ere

\section{Connection between the monomial and Gelfand--Tsetlin-type bases}
\label{sec:tm}

We will demonstrate that the transition matrix between the basis $\theta^{}_{\La}$
of the $\spa_{2n}$-module $V(\la)$ provided by Theorem B and the basis $\zeta^{}_{\La}$
of Theorem~\ref{thm:genactc} is triangular.

Using the notation from the previous section,
for each $n$-tuple $\nu$ satisfying the betweenness conditions, introduce
the vector $\eta^{}_{\nu}\in V(\la)^+_{\mu}$ by
\beql{cxinueta}
\eta^{}_{\nu}=
\prod_{i=1}^{n-1}z_{n,-i}^{\mu_{i}-\nu_{i+1}}
z_{-i,-n}^{\la_{i}-\nu_{i+1}}\ts
F_{n,-n}^{{\ts}-\nu_1}\ts\xi,
\eeq
where we let $\xi$ denote the highest weight vector of $V(\lambda)$.
By a result of Zhelobenko~\cite[Theorem~6.1]{z:gz}, the vectors $\eta^{}_{\nu}$
form a basis of $V(\la)^+_{\mu}$. This fact will also follow from a relationship
between the vectors $\xi^{}_{\nu}$ and $\eta^{}_{\nu}$ as described in the next lemma.
We will consider the lexicographical orderings $\succ$ on the set of $n$-tuples $\nu$
and on the set of $(n-1)$-tuples $\mu$.

\ble\label{lem:trilomo}
For any $\nu$ we have the relation
\ben
\eta^{}_{\nu}=\sum_{\nu'\succcurlyeq \nu} c_{\nu,\nu'}\tss \xi^{}_{\nu'}
\een
for some constants $c_{\nu,\nu'}$, and $c_{\nu,\nu}\ne 0$.
In particular, the vectors $\eta^{}_{\nu}$ form a basis of $V(\la)^+_{\mu}$.
\ele

\bpf
Since $F_{n,-n}$ commutes with the lowering operators $z_{n\ts j}$, the vector
\eqref{cxinueta} can be written as $\eta^{}_{\nu}=F_{n,-n}^{{\ts}-\nu_1}\tss \xi^{}_{\nu^{0}}$,
where $\nu^{0}$ is the $n$-tuple obtained from $\nu$ by replacing $\nu_1$ with $0$.
On the other hand, by the formulas of Theorem~\ref{thm:genactc} for any $\nu$ we have
\ben
F_{n,-n}\ts\xi_{\nu}=\sum_{i=1}^n\prod_{a=1,\ts a\ne i}^n
\frac{1}{\ga_i^2-\ga_a^2}\ \xi_{\nu-\de_i}.
\een
A repeated application of this formula allows us to write
$F_{n,-n}^{{\ts}-\nu_1}\tss \xi^{}_{\nu^{0}}$ as a linear combination
of the basis vectors $\xi_{\nu}$ which clearly has the required form.
\epf

Lemma~\ref{lem:trilomo} implies that the vectors
\beql{etala}
\eta^{}_{\La}=\prod_{k=1,\dots,n}^{\longrightarrow}
\Bigg(F_{k,-k}^{{\ts}-\la'_{k\tss 1}}\prod_{i=1}^{k-1}
z_{k,-i}^{\la^{}_{k-1\ts i}-\la'_{k\ts i+1}}
z_{-i,-k}^{\la^{}_{k\tss i}-\la'_{k\ts i+1}}
\Bigg)\ts\xi,
\eeq
parameterised by all  type {\sf C} patterns $\La$ associated with $\la$
form a basis of the representation $V(\la)$.

Since the weight $\mu$ will now be varied, we will denote the vector
\eqref{cxinueta} by $\eta^{}_{\nu\mu}$.
The following lemma is essentially a particular case of
\cite[Theorem~7]{z:gz} or \cite[Lemma~2]{z:ep}.

\ble\label{lem:mscomp}
For any given pair $(\nu,\mu)$
satisfying the betweenness conditions,
in the module $V(\la)$ we have
\beql{transi}
F_{n,-n}^{{\ts}-\nu_1}\ts
\prod_{i=1}^{n-1}F_{n,-i}^{\mu_{i}-\nu_{i+1}}
F_{-i,-n}^{\la_{i}-\nu_{i+1}}\ts
\xi =
c\ts\tss \eta^{}_{\nu\mu}
+\sum_{\nu',\ts\mu'}\ts u(\nu',\mu')\ts
\eta_{\nu'\mu'}
\eeq
for a nonzero constant $c$ and some elements $u(\nu',\mu')\in\U(\n^-_0)$,
where the sum is taken over the pairs $(\nu',\mu')$
satisfying the betweenness conditions, and $u(\nu',\mu')=0$ unless $\mu'\succ\mu$, or
$\mu'=\mu$ and $\nu'\succ\nu$.
\ele

\bpf
Write the product on the left hand side in the order
\ben
F_{n,-n+1}^{\mu_{n-1}-\nu_{n}}\dots F_{n,-1}^{\mu_{1}-\nu_{2}}\ts F_{n,-n}^{{\ts}-\nu_1}\ts
F_{-1,-n}^{\la_{1}-\nu_{2}}\dots F_{-n+1,-n}^{\la_{n-1}-\nu_{n}}\ts
\xi.
\een
Taking into account that
$F_{n,-k}=F_{k,-n}$ for positive values of $k$, start from the rightmost
generator and proceed to the left by using the
{\em inversion formula} \cite[Lemma~9.2.2]{m:yc} to replace $F_{i,-n}$
with $i=-n+1,\dots,n-1$
by the expression:
\ben
F_{i,-n}=pF_{i,-n}+\sum_{i>i_1>\dots>i_s>-n}
F_{i\tss i_1}F_{i_1i_2}\dots F_{i_{s-1}i_s}\ts pF_{i_s,-n}
\frac{1}{(f_{i_s}-f_{i})(f_{i_s}-f_{i_1})\dots (f_{i_s}-f_{i_{s-1}})},
\een
summed over $s=1,2,\dots$.
Apply relation \eqref{normali} to write the right hand side of
the inversion formula in terms of the lowering operators $z_{k,-n}$.
We will use the following property of these
operators: $z_{i,-n}$ and $z_{j,-n}$ commute for $i+j\ne 0$;
see \cite[Proposition~9.2.5]{m:yc}.
Let ${\tilde\n}^-_0$ denote the subalgebra of $\n^-_0$ spanned by the elements $F_{j\tss i}$
with $1\leqslant i<j\leqslant n-1$.
The same argument as in the proof of Lemma~\ref{lem:hwcomp}
shows that
\ben
F_{n,-n}^{{\ts}-\nu_1}\ts F_{-1,-n}^{\la_{1}-\nu_{2}}\dots F_{-n+1,-n}^{\la_{n-1}-\nu_{n}}\ts
\xi
= d\ts\tss \eta_{\nu\tilde\nu}
+\sum_{\si\succ\nu}\ts u(\si)\ts
\eta_{\si\tilde\nu}
\een
for a nonzero constant $d$ and some elements $u(\si)\in\U({\tilde\n}^-_0)$, where
$\tilde\nu=(\nu_2,\dots,\nu_n)$. Now we will be applying the inversion formula
for positive values of $i$ and note that each term with $i_s<0$ in the sum
on the right hand side contains a generator $F_{i_k\tss i_{k+1}}$ with $i_k>0>i_{k+1}$.
However, such a generator commutes with all elements $F_{i,-n}$ for $i>0$.
Therefore, all these terms with $i_s<0$ will only contribute to the sum on the right hand side
of the expansion \eqref{transi} within the summands of the form $u(\nu',\mu')\ts\eta_{\nu'\mu'}$
with $\mu'\succ\mu$.

On the other hand, for any element $u\in\U({\tilde\n}^-_0)$ we have the relation
\ben
F_{i,-n}\tss u=u\tss F_{i,-n}+\sum_{j=i+1}^{n-1}\tss F_{j,-n}\tss u_j
\een
for certain elements $u_j\in\U({\tilde\n}^-_0)$. Hence, considering
the terms in the inversion formula with the property $i_s>0$, we may conclude that
nonzero summands on the right hand side
of \eqref{transi} of the form $u(\nu',\mu)\ts\eta_{\nu'\mu}$ must have the property
$\nu'\succcurlyeq\nu$ and $u(\nu,\mu)$ is a nonzero constant.
\epf

Consider the
vectors $\xi^{}_{\Lambda}\in V(\la)$ introduced in Section~\ref{sec:wb}.
They are parameterised by the type {\sf C}  patterns $\La$ defined in the Introduction.
Represent each pattern $\La$ associated with $\la$ as the sequence of the rows:
\ben
\La=(\bar\la^{}_{n-1},\bar\la'_{n},
\bar\la^{}_{n-2},\bar\la'_{n-1},\dots,\bar\la'_{1}),\qquad
\een
where we set
\ben
\bar\la_{k}=(\la_{k\tss 1},\dots,\la_{k\tss k})\Fand \bar\la'_{k}=(\la'_{k\tss 1},\dots,\la'_{k\tss k}).
\een
Introduce the lexicographical ordering $\succ$ on the sequences $\La$ by using
the lexicographical orderings on the vectors $\bar\la_{k}$ and
$\bar\la'_{k}$.
Recall the vectors $\theta^{}_{\La}$ defined  in Theorem B. We can now obtain another proof
of the theorem.

\bpr\label{prop:triangc}
For each  type {\sf C} pattern $\La$ associated with $\la$,
in the module $V(\la)$ we have
\ben
\theta^{}_{\La}=\sum_{\La'\succcurlyeq \La} c_{\La,\La'}\tss \xi^{}_{\La'}
\een
for some constants $c_{\La,\La'}$, and $c_{\La,\La}\ne 0$.
In particular, $\theta^{}_{\La}$ is a basis of $V(\la)$.
\epr

\bpf
We will use an induction on $n$. Consider the part of the product
defining the vector $\theta^{}_{\La}$
which corresponds to the value $k=n$. By applying
Lemma~\ref{lem:mscomp} and using the induction hypothesis, we can write
$\theta^{}_{\La}$ as a linear combination
of the basis vectors $\eta^{}_{M}$ defined in
\eqref{etala} so that it contains the vector $\eta^{}_{\La}$
with a nonzero coefficient, while the remaining vectors occurring in the
linear combination have the property $M\succ \La$. It remains
to expand the vectors $\eta^{}_{M}$ as linear combinations of basis vectors
$\xi^{}_{\La'}$ by using Lemma~\ref{lem:trilomo} which yields the
expansion of $\theta^{}_{\La}$ with the required properties.
\epf

\begin{rmk}
The inversion formula can be  used  also for rewriting the basis  of Proposition~\ref{sp-prop} in terms of
the lowering operators.   Therefore  the subspace
$V(\lambda)^+$ has a basis
$$
\{ F_{2n\,1}^b z_{2n\,n}^{b_n+\iota_n a_n} z_{n\,1}^{b_n+(1-\iota_n)a_n} \ldots z_{2n\,2}^{b_2+\iota_2a_2} z_{2\,1}^{b_2+(1-\iota_2)a_2}  v_\lambda
\mid a_2,\ldots a_n, b, b_2,\ldots b_n \text{ satisfy } \eqref{1}-\eqref{4}
\},
$$
where $\iota_k\in\{0,1\}$.
\end{rmk}

\end{document}